\documentclass[reqno]{amsart}
\usepackage[utf8]{inputenc}
\usepackage{xcolor,fullpage,hyperref,cleveref,enumitem,enumitem}
\usepackage{tikz,amsmath,amssymb,color,mathtools}
\usepackage{amsthm}
\usetikzlibrary{calc}
\usepackage{float}
\usepackage{pgf}
\usepackage[backrefs]{amsrefs}

\usepackage[normalem]{ulem}

\newcommand{\X}{\texttt{X}}
\newcommand{\Pref}{\mathrm{Pref}}

\newcommand{\Sym}{\mathfrak{S}}
 
\newcommand{\Des}{\mathrm{Des}}
\newcommand{\DB}{\mathrm{DBottom}}

\newcommand{\Cat}{\mathrm{Cat}}

\theoremstyle{definition} 
\newtheorem{theorem}{Theorem}[section]
\newtheorem{corollary}{Corollary}[theorem]
\newtheorem{proposition}[theorem]{Proposition}

 \newtheorem{example}{Example}
 
\newtheorem{lemma}[theorem]{Lemma}
\newtheorem{claim}{Claim}

\theoremstyle{definition}
\newtheorem{definition}{Definition}[section]
\newcommand{\Lucky}{\mathsf{Lucky}} 
\newcommand{\lucky}{\mathsf{lucky}}  
\newcommand{\LPF}{\mathsf{LuckyPF}}

\theoremstyle{remark}

\newcommand{\PF}{\mathrm{PF}}

\newcommand{\out}{\mathcal{O}}

\newcommand{\seqnum}[1]{\href{https://oeis.org/#1}{\rm  \underline{#1}}}
\usepackage{thmtools} 
\usepackage{thm-restate}

\title{Parking functions with a fixed set of lucky cars}

\author[Harris]{Pamela E. Harris}
\address[P.~E.~Harris]{Department of Mathematical Sciences, University of Wisconsin-Milwaukee, Milwaukee, WI 53211}
\email{\textcolor{blue}{\href{mailto:peharris@uwm.edu}{peharris@uwm.edu}}}

\author[Martinez]{Lucy Martinez}
\address[L.~Martinez]{Department of Mathematics, Rutgers University, Piscataway, NJ 08854}
\email{\textcolor{blue}{\href{mailto:lucy.martinez@rutgers.edu}{lucy.martinez@rutgers.edu}}}

\begin{document}
\begin{abstract}
    In a parking function, a lucky car is a car that parks in its preferred parking spot and the parking outcome is the permutation encoding the order in which the cars park on the street. 
    We give a characterization for the set of parking outcomes arising from parking functions with a fixed set of lucky cars. This characterization involves the descent bottom set of a permutation, and we use the characterization to give a formula for the number of parking functions with a fixed set of lucky cars. 
    Our work includes the cases where the number of cars is equal to the number of parking spots, and where there are more spots than cars. 
    We also give product formulas for the number of weakly increasing parking functions having a fixed set of lucky cars, and when the number of cars equals the number of spots this is a product of Catalan numbers. 
\end{abstract}
\maketitle

\section{Introduction}
Throughout we let $n\in\mathbb{N}\coloneq\{1,2,3,\ldots\}$ and  $[n]\coloneq\{1,2,\ldots,n\}$. The set of parking functions of length~$n$ consists of all $n$-tuples $\alpha=(a_1,a_2,\ldots,a_n)\in[n]^n$ whose rearrangement into weakly increasing order $\alpha^\uparrow=(b_1,b_2,\ldots,b_n)$ satisfies $b_i\leq i$ for all $i\in[n]$.
The use of the name ``parking function'' arises as these $n$-tuples encode the parking preferences of $n$ cars attempting to park on a one-way street consisting of $n$ parking spots. The cars enter the street in sequential order, from car $1$ to car $n$, with car $i$ parking in the first spot it encounters at or past its preference $a_i$. 
If $\alpha=(a_1,a_2,\ldots,a_n)$ denotes the preference list for all cars and all cars are able to park in the $n$ spots on the street, then $\alpha$ is a \textit{parking function of length} $n$, and the order in which the cars park on the street is called the \textit{outcome} of $\alpha$ and is denoted $\out(\alpha)$. More precisely, if $\Sym_n$ denotes the set of permutations of $[n]$ written in one-line notation, then the outcome of $\alpha$~is
\[\out(\alpha)=\pi_1\pi_2\cdots\pi_n,\]
where $\pi_i$ denotes that car $\pi_i$ parked in spot $i$ on the street. We let $\PF_n$ denote the set of parking functions of  length $n$ and Konheim and Weiss established $|\PF_n|=(n+1)^{n-1}$, see \cite{konheim1966occupancy}.

In a parking function, a car is said to be \textit{lucky} if it parks in its preferred parking spot. 
For example, the parking function $(1,1,1,1)$ has outcome $\out((1,1,1,1))=1234$
and only the first car is lucky. If the entries in a parking function are all distinct, i.e., $\alpha=(\pi_1,\pi_2,\ldots,\pi_n)$, where $\pi=\pi_1\pi_2\cdots\pi_n$ is a permutation of $[n]$, then $\out(\alpha)=\pi^{-1}$ and all cars are lucky. 
Gessel and Seo gave a product formula for the generating function for the number of parking functions based on the number of lucky cars \cite{GesselandSeo} which is given by 
\begin{equation}
    \label{eq:GesselSeo}
  L_n(q)=\sum_{\alpha \in \PF{n} }q^{\lucky(\alpha)} = q\prod_{i=1}^{n-1} (i+(n-i+1)q).
\end{equation}
Gessel and Seo established Equation \eqref{eq:GesselSeo} via ordered set partitions and generating function techniques. Shin gave a bijective proof of Gessel and Seo's result by constructing a nonrecursive bijection between forests and
parking functions \cite{Heesung}. 

Other work related to the lucky statistic includes that of Harris, Kretschmann, and Mart\'inez Mori who showed that the number of parking preference lists, elements of $[n]^n$, with exactly $n-1$ lucky cars
is the same as the
total number of comparisons performed by the \texttt{Quicksort} algorithm given all possible orderings of an array of size $n$ \cite{harris2023lucky}. 
Also Aguillon,
Alvarenga, Harris, Kotapati, Martínez Mori, Monroe, Saylor, Tieu, and  Williams II~\cite{displacement_hanoi} established that the set of parking functions with exactly one unlucky car parking one spot away from its preference is in bijection with the set of ideal states in the tower of Hanoi game \cite[Theorem 1]{displacement_hanoi}.
Colmenarejo, Dawkins, Elder, Harris, Harry, Kara, Smith, and Tenner \cite{colmenarejo2024luckydisplacementstatisticsstirling} studied Stirling permutations as parking functions and considered the questions of what subsets of cars are lucky, and how many Stirling permutations have certain sets of lucky cars. 
They establish that there are $(n-1)!$ Stirling permutations with only the first car being lucky, and that there are Catalan many Stirling permutations with a  maximal number of lucky cars. 
They also study cases in which a small number of cars are lucky, however general results for counting Stirling permutations with a fixed set of lucky cars remains an open problem.
These examples illustrate that the lucky statistic of parking functions give rise to interesting connections with other well-studied  families of combinatorial objects.

We study parking functions with a fixed set of lucky cars. To make our approach precise, we let $\Lucky(\alpha)$ denote the set of lucky cars of $\alpha\in\PF_n$. 
For example, $\Lucky((1,1,1,1))=\{1\}$, while $\Lucky(\alpha)=[n]$ whenever $\alpha=(\pi_1,\pi_2,\ldots,\pi_n)$ and $\pi=\pi_1\pi_2\cdots\pi_n\in\Sym_n$.
As expected, the first car in any parking function is always lucky, so $1\in\Lucky(\alpha)$ for all $\alpha\in \PF_n$. 
For any subset $I\subseteq[n]$, we define 
\[\LPF_n(I)\coloneq\{\alpha\in \PF_n: \Lucky(\alpha)=I\},\]
which is the set of parking functions of length $n$ with lucky cars being the cars in the set $I$ and unlucky cars being the cars in the set $[n]\setminus I$.
For $I\subseteq[n]$, if $\LPF_n(I)\neq
\emptyset$, then $I$ is said to be a \textit{lucky set} of $\PF_n$. 
For example, as the first car is always a lucky car regardless of the preference of the other cars, a set $I$ not containing $1$ can never be a lucky set of $\PF_n$ for any $n\geq 1$. 
We stress that  lucky sets contain the lucky cars rather than the position at which the lucky cars park.

Among our contributions,
in Theorem~\ref{thm:lucky sets}, we characterize the lucky sets of $\PF_n$ by proving that a subset $I\subseteq[n]$ is a lucky set of $\PF_n$ if and only if $1\in I$. 
An immediate consequence is that there are $2^{n-1}$ distinct lucky sets of $\PF_n$ (Corollary~\ref{cor:number lucky sets}). 
Curiously, not every permutation can be an outcome of a parking function with a fixed set of lucky cars. 
For example, the  permutation $\pi=312\in \Sym_3$ cannot be the outcome of a parking function $\alpha\in\PF_3$ with lucky set $\{1\}$. 
This is because car $3=\pi_1$ parks in spot $1$, however when car $3$ enters to park, spot $1$ is unoccupied. 
This means that parking in spot $1$ would make car $3$ a lucky car, but we assumed that the set of lucky cars was $\{1\}$, reaching a contradiction.
Based on this observation, a natural question to ask is: 
\emph{For a fixed subset  $I\subseteq[n]$ (containing $1$), what permutations in $\Sym_n$ arise as the outcome of parking functions with lucky set $I$?}
We answer this question fully via the following characterization.

\setcounter{section}{2} 
\setcounter{theorem}{2}
\begin{theorem}
Fix a lucky set $I\subseteq[n]$ of $\PF_n$. 
Then $\pi=\pi_1\pi_2\cdots\pi_n\in \Sym_n$ is the outcome $\alpha\in\LPF_n(I)$
if and only if 
\begin{enumerate}
    \item $\pi_1\in I$,
    \item if $\pi_i\in I$ and $\pi_{i+1}\notin I$, then $\pi_i<\pi_{i+1}$, and 
    \item if $\pi_{i-1}>\pi_{i}$, then $\pi_{i}\in I$.
\end{enumerate}
\end{theorem}

The proof of Theorem \ref{thm:characterizing outcomes} involves descent bottom sets of permutations (Definition \ref{def:descent bottoms}), which are the values $\pi_{i+1}$, whenever $\pi_i>\pi_{i+1}$. 
Any  descent bottom in an outcome of a parking function would imply that car $\pi_{i+1}$ is lucky since it parked in spot $i+1$ having found spot $i$ vacant, as  car $\pi_i$ had not yet entered the street to park.
 We let $\out_n(I)$ denote the set of outcome permutations for parking functions with lucky set $I$.
  If $I=\{1\}$, then $\pi=123\cdots n$ is the only possible outcome of a parking function having only car~$1$ being lucky. Hence $|\out_n(I)|=1$. However, even though Theorem \ref{thm:characterizing outcomes} gives a characterization for the set $\out_n(I)$, determining a formula for the cardinality of the set $\out_n(I)$ has proven difficult for general lucky set~$I$. 
 We give a result for when the first $k$ cars are lucky.
 \begin{theorem}\label{thm:counting outcomes with first k cars lucky}
     Let $I=\{1,2,3,\ldots,k\}\subseteq[n]$ be a lucky set of $\PF_n$. Then 
     \[|\out_n(I)|=\sum_{J=\{j_1=1,j_2,j_3,\ldots,j_k\}\subseteq[n]}k!\binom{n-k}{j_2-j_1-1,j_3-j_2-1,\ldots,j_{k}-j_{k-1}-1,n-j_{k}}.\]
 \end{theorem}
The proof of \Cref{thm:counting outcomes with first k cars lucky} relies on  specifying the spots that are occupied by the lucky cars, which must always include the first spot, and noting that those cars can park in those spots in any order. 

In our main result, we utilize the set of outcomes $\out_n(I)$ to give
a formula for the number of parking functions of length $n$ with lucky set $I$. 

\setcounter{theorem}{5}
\begin{theorem}
Fix a lucky set $I\subseteq[n]$ of $\PF_n$. Let $\out_n(I)$ denote the set of outcome permutations for parking functions with lucky set $I$.
Then the number of parking functions with lucky set $I$ is
\begin{align}
|\LPF_n(I)|=
\sum_{\pi=\pi_1\pi_2\cdots\pi_n\in \out_n(I)}
\left(\prod_{\pi_j\notin I} \ell (\pi_j)\right),
\label{eq:main counting sum}
\end{align}
 where, for each $i\in[n]$, 
 $\ell(\pi_i)$
    denotes the length of the longest subsequence $\pi_j\pi_{j+1}\cdots\pi_{i-1}$ with $\pi_t<\pi_i$ for all $j\leq t\leq i-1$.
\end{theorem}

The proof of Theorem \ref{thm:count pfs with fixed lucky set} relies on the characterization of the permutations that are outcomes of parking functions with lucky set $I$ (as described in \Cref{thm:characterizing outcomes}), and
an argument often referred to as ``counting through permutations''. This technique allows one to count the number of parking functions with a given parking outcome. For other work that uses this technique for counting parking functions we refer the reader to  \cite{franks2023counting,colmenarejo2021counting,lintervalrational,Spiro}. Moreover, taking \Cref{eq:main counting sum} as a sum over lucky sets of fixed cardinality $k$ yields the coefficient of $q^k$ in \Cref{eq:GesselSeo}. That is, for a fixed value $1\leq k \leq n$,  we have that 
\[[q^k]\left(q\prod_{i=1}^{n-1} (i+(n-i+1)q)\right)=\sum_{I\subseteq[n],1\in I, |I|=k}\left(\sum_{\pi=\pi_1\pi_2\cdots\pi_n\in \out_n(I)}
\left(\prod_{\pi_j\notin I} \ell (\pi_j)\right)\right).\]

We also consider the subset $\PF_n^\uparrow$ of weakly increasing parking functions.
The only possible outcome of a weakly increasing parking function is the identity permutation $\pi=123\cdots n$. Hence, restricting to the set $\PF_n^\uparrow$ reduces  
 \Cref{eq:main counting sum} to just the product, which we show is a product of Catalan numbers \cite[\href{https://oeis.org/A000108}{A000108}]{OEIS}. 
 For $n\geq 1$, the $n$th Catalan number is $\Cat_n=\frac{1}{n+1}\binom{2n}{n}$. 
We can now state the result.
\setcounter{theorem}{7}
\begin{theorem}
Fix a lucky set $I=\{i_1,i_2,i_3,\ldots,i_k\}\subseteq[n]$ of $\PF_n^{\uparrow}$ and assume $1=i_1<i_2<\cdots<i_k\leq n$.
If $\LPF_n^{\uparrow}(I)$ denotes the set of weakly increasing parking functions in $\PF_n^\uparrow$ with lucky set $I$, then 
\begin{align}
|\LPF_n^{\uparrow}(I)|=\prod_{j=1}^k\Cat_{x_j},
\end{align}    
where 
$x_{j}=i_{j+1}-i_j-1$ for each $j\in [k-1]$ and $x_k=n-i_k$.
\end{theorem}

In Section \ref{sec:more spots than cars}, we generalize all of the aforementioned results to the set of $(m,n)$-parking functions, which are parking functions with $m$ cars and $n$ parking spots with $m\leq n$. 
We characterize the lucky sets (\Cref{thm:lucky sets of mn}) and characterize the outcomes of $(m,n)$-parking functions with a fixed set of lucky cars (\Cref{thm:characterizing outcomes in mn}). 
We also give a formula for the number of outcomes for $m$ cars parking in $n$ spots with the first $k$ cars being lucky (Theorem \ref{thm:number of outcomes with first k cars lucky in mn}).
For $n\geq m$, 
we give a formula for the number of $(m,n)$-parking functions  a fixed set of lucky cars (\Cref{thm:count pfs with fixed lucky set - mn version}). 
We also give a product formula for the number of weakly increasing $(m,n)$-parking functions having lucky set $I$ (\Cref{thm:weakly increasing pfs with a fixed set of lucky cars - mn version}). 
We conclude with Section~\ref{sec:future}, where we state a few open problems for further study.

\setcounter{section}{1}
\section{Same number of cars as spots}\label{sec:background}
In this section, we begin by describing all of the possible lucky sets. We then  characterize the types of permutations that are outcomes of parking functions with a fixed set of lucky cars, and then we give a formula for the number of parking functions with a fixed set of lucky cars.

\subsection{Lucky sets}
We begin by establishing necessary and sufficient conditions on the set $I\subseteq[n]$ so that it is a lucky set of $\PF_n$ i.e., it is a set for which there is a parking function with $I$ being the cars that are lucky and $[n]\setminus I$ is the set of unlucky cars.

\begin{theorem}\label{thm:lucky sets}
A subset $I\subseteq[n]$ is a lucky set of $\PF_n$ if and only if $1\in I$.
\end{theorem}
\begin{proof}
    $(\Rightarrow)$ If $I$ is a lucky set of $\PF_n$, then there exists a parking function $\alpha\in\PF_n$ such that $\Lucky(\alpha)=I$. 
    As the first car is always a lucky car for any $\alpha\in \PF_n$, this implies that $1\in I$.

    \noindent $(\Leftarrow)$ Assume $I=\{i_1=1,i_2,i_3,\ldots,i_k\}$, with $k\in[n]$.
    Without loss of generality assume that $i_1=1<i_2<i_3<\cdots<i_k$.
    Let $\alpha=(a_1,a_2,\ldots,a_n)$ where
    \[a_x=\begin{cases}
        x &\mbox{if $x\in I$}\\
        i_\ell &\mbox{if $i_\ell<x<i_{\ell+1}$ .}
    \end{cases}\]
    By construction $\alpha$ is weakly increasing and satisfies that $a_i\leq i$ for all $i\in [n]$. Thus $\alpha\in\PF_n$.
    Every weakly increasing parking function parks the cars in sequential order, ie. car $i$ parks in spot $i$ for all $i\in[n]$.  
    For all $x\in I$, car $x$ has preference $x$, and parks in spot $x$. Thus every car in $I$ is a lucky car of $\alpha$. 
    On the other hand, 
    if $x\notin I$, then car $x$ has preference strictly smaller than $x$ and parks in spot $x$. Thus, these cars would be unlucky.
    Therefore, $\Lucky(\alpha)=I$ and $I$ is a lucky set of $\PF_n$, as desired.
\end{proof}

Theorem~\ref{thm:lucky sets} immediately implies the following corollaries.

\begin{corollary}\label{cor:number lucky sets}
For $n\geq 1$, there are $2^{n-1}$ distinct lucky sets of $\PF_n$.    
\end{corollary}
\begin{proof}
    Since the first car is always lucky, then, for each $1<i\leq n$, it is only a matter of whether car $i$ is lucky or not. This yields $2^{n-1}$ lucky sets of $\PF_n$, as claimed.
\end{proof}

\begin{corollary}
    If $I\subseteq[n]$ is a lucky set of $\PF_n$, then it is a lucky set of $\PF_m$ for all $m\geq n$.
\end{corollary}
\begin{proof}
        If $\alpha=(a_1,a_2,\ldots,a_n)$ is such that $\Lucky(\alpha)=I$, then appending $m-n$ copies of the value $n$ to $\alpha$ satisfies that $\Lucky((a_1,a_2,\ldots,a_n,\underbrace{n,n,\ldots,n}_{m-n\mbox{ times}}))=I$.
\end{proof}

\subsection{Outcomes of parking functions with a fixed lucky set}\label{sec:outcomes}
Not every permutation is the outcome of a parking function with a fixed set of lucky cars.

\begin{example}
    Let $\alpha\in \PF_4$ and suppose that $\Lucky(\alpha)=\{1,4\}$. Notice that $\pi=1423\in \Sym_4$ is not a possible outcome since car $2$ would park in spot $3$, and would do so having found spot $2$ unoccupied on the street, as it attempts to park. Thus, the only way car $2$ would park in spot $3$ is if it preferred that spot, in which case car $2$ would be lucky, implying that $2\in\Lucky(\alpha)$, a contradiction.
\end{example}

The previous example illustrates that if a permutation is to be the outcome of a parking function with a fixed set of lucky cars, then the descents in that permutation force certain cars to be lucky. 
Before proving this, we adapt the usual definition of a descent and descent bottom for our purposes. 

\begin{definition}\label{def:descent bottoms}
    Given a permutation $\pi=\pi_1\pi_2\cdots\pi_n\in \Sym_n$, an index $1<i\leq  n$ is a \textit{descent} of $\pi$ if $\pi_{i-1}>\pi_i$, and the value $\pi_i$ is called a \textit{descent bottom} of $\pi$. 
    For convenience, we say $i=1$ is always a descent and $\pi_1$ is always a descent bottom.
    The set
    \begin{align*}
    \Des(\pi)&\coloneq\{i\in [n]: i \mbox{ is a descent of }\pi\},
    \intertext{
    is called the \textit{descent set} of $\pi$, and the set }
    \DB(\pi)&\coloneq\{\pi_i\in [n]: \pi_i \mbox{ is a descent bottom of }\pi\},
\end{align*}
is called the \textit{descent bottom set} of $\pi$.
\end{definition}
By using generating functions and well-known bijections between parking functions, Pr\"{u}fer codes, and labeled Dyck paths, Schumacher gave explicit formulas for the total number of descents among parking functions of length $n$ with exactly $i$ ties  \cite{Schumacher_Descents_PF}. Cruz, Harris, Harry, Kretschmann, McClinton, Moon, Museus, and Redmon \cite{cruz2024discretestatisticsparkingfunctions} gave a recursive formula for the number of parking functions with a given descent set $I\subseteq[n]$. 

In our work, we use the descent bottom set of a permutation in connection with describing parking functions with a fixed set of lucky cars.
We begin by establishing that 
if $\pi\in\Sym_n$
is the outcome of a parking function $\alpha $ with $I$ as a lucky set,
then 
all descent bottoms of $\pi$ must be lucky cars of $\alpha$.

\begin{lemma}\label{lem:descent bottoms are lucky}
    Fix $\pi=\pi_1\pi_2\cdots\pi_n\in \Sym_n$. For any $\alpha\in \PF_n$ with outcome $\out(\alpha)=\pi$, 
    if $i\in \Des(\pi)$, then  $\pi_i\in \Lucky(\alpha)$. In other words,
$\DB(\pi)\subseteq\Lucky(\alpha)$.
\end{lemma}
\begin{proof}
Let $\alpha$ be a parking function of length $n$  with outcome $\out(\alpha)=\pi=\pi_1\pi_2\cdots\pi_n\in\Sym_n$.
For each $i\in[n]$, since $\out(\alpha)=\pi=\pi_1\pi_2\cdots\pi_n$ we know that car $\pi_i$ parked in spot $i$.
If $i=1$, then car $\pi_1$ was the first car with preference for spot $1$, and parked there, hence it is a lucky car. Implying that ~$\pi_1\in\Lucky(\alpha$).
By definition, for any $1<i\leq n$, if  $i\in\Des(\pi)$, then $\pi_{i-1}>\pi_i$.
This implies that car 
$\pi_{i}$ parks on the street before car $\pi_{i-1}$ enters the street to park. 
Thus, when car $\pi_i$ entered the street to park, spot $i-1$ is unoccupied. 
Assume for contradiction that car $\pi_i$ has a preference $1\leq a_i\leq i-1$. Then car $\pi_i$ would park in the first available spot past $a_i$. 
Among the spots numbered $a_i,a_i+1,\ldots, i-1$ spot $i-1$ remains available. 
So car $\pi_i$ would park in spot $x$ satisfying $a_i\leq x\leq i-1<i$, contradicting the assumption that car $\pi_i$ parks in spot $i$. 
Thus, car $\pi_i$ parks in spot $i$, must have done so by preferring spot $i$. 
This makes car $\pi_i$ a lucky car, 
and so whenever $i\in \Des(\pi)$ we have that $\pi_i\in \Lucky(\alpha)$, 
as claimed.
\end{proof}

Lemma~\ref{lem:descent bottoms are lucky} establishes that descent bottoms are lucky cars. 
However, not all lucky cars have to be descent bottoms. 
For example, if $I=[n]$, then all cars are lucky and a possible outcome is $\pi=123\cdots n$, which has $\pi_1=1$ as the only descent bottom.
We now give a complete characterization for permutations that are outcomes of a parking function with a fixed set of lucky cars.

\begin{theorem}\label{thm:characterizing outcomes}
Fix a lucky set $I\subseteq[n]$ of $\PF_n$. 
Then $\pi=\pi_1\pi_2\cdots\pi_n\in \Sym_n$ is the outcome an $\alpha\in\LPF_n(I)$
if and only if 
\begin{enumerate}
    \item $\pi_1\in I$,
    \item if $\pi_i\in I$ and $\pi_{i+1}\notin I$, then $\pi_i<\pi_{i+1}$, and 
    \item if $\pi_{i-1}>\pi_{i}$, then $\pi_{i}\in I$.
\end{enumerate}
\end{theorem}
\begin{proof}
\noindent $(\Rightarrow)$ For the forward direction we verify that when $\pi$ is an outcome of a parking function with lucky set $I$ it satisfies  each of the needed conditions.
\begin{itemize}[leftmargin=.6in]
\item[For (1):] In any parking function the first car to prefer spot $1$ will park there, hence, that car is a lucky car. 
So if $\pi=\pi_1\pi_2\cdots\pi_n=\out(\alpha)$, then $\pi_1\in I$.

\item[For (2):] Suppose $\pi_i\in I$ and $\pi_{i+1}\notin I$. 
If $\pi_i<\pi_{i+1}$, then we are done. For sake of contradiction, suppose that  $\pi_i>\pi_{i+1}$. 
This means that when car $\pi_{i+1}$ parks in spot $i+1$, car $\pi_i$ has not yet entered the street to park. 
Hence, spot $i$ is empty, when car $\pi_{i+1}$  enters the street to park. 
However, car $\pi_{i+1}$ is not lucky, so it must have preferred a spot numbered $x$, where $x<i+1$. 
As spot $i$ is empty, car $\pi_{i+1}$ would have parked in a spot numbered $y$ with $y\leq i$. 
Hence, car $\pi_{i+1}$ would not have parked in spot $i+1$, a contradiction. 

 \item[For (3):] Suppose that $\pi_{i-1}>\pi_{i}$. Then, by  Lemma~\ref{lem:descent bottoms are lucky}, $\pi_i\in I$.
\end{itemize}
\smallskip

\noindent$(\Leftarrow)$
Let $\pi=\pi_1\pi_2\cdots\pi_n\in \Sym_n$ 
satisfy conditions (1), (2) and (3) above. 
We use the following notation:
\begin{itemize}
    \item $\pi_x=y$ denotes that car $y$ parked in spot $x$ on the street,
    \item $\pi_u^{-1}=v$ denotes that car $u$ parks in spot $v$ on the street.
\end{itemize}
By the same argument as in \cite[Theorem 2.2]{MVP}, if $\pi^{-1}=\pi^{-1}_1\pi^{-1}_2\cdots\pi^{-1}_n$
, then \[\out{}((\pi^{-1}_1,\pi^{-1}_2,\ldots,\pi^{-1}_n))=\pi.\]
We now construct $\alpha=(a_1,a_2,\ldots,a_n)\in \PF_n$, such that $\out(\alpha)=\pi$ and $\Lucky(\alpha)=I$.
For each $\pi_i=1,2,\ldots, n$ (in this order) set the value of $a_{\pi_i}$ as follows:

\begin{enumerate}
    \item If $\pi_i\in I$, set $a_{\pi_i}=\pi_{j}^{-1}$ where $j=\pi_i$.
    \item If $\pi_i\notin I$, set $a_{\pi_i}$ as follows:
    \begin{enumerate}
        \item If $\pi_{i-1}\in I$ and $\pi_{i-1}<\pi_i$, then set $a_{\pi_i}=\pi_{k}^{-1}$ where $k=\pi_{i-1}$.
        
        \item If $\pi_{i-1}\in I$ and $\pi_{i-1}>\pi_i$, then by Lemma \ref{lem:descent bottoms are lucky} we would have that  $\pi_{i}\in I$, a contradiction. Thus this case never happens.

        \item If $\pi_{i-1}\notin I$, then set $a_{\pi_{i}}=\pi_{k}^{-1}$ where $k=\pi_{i-1}$.
    \end{enumerate}
\end{enumerate}
We must now establish the following claims:\\

    \begin{claim}   
    As constructed, $\alpha$  is a parking function.
    \end{claim} 
    \begin{proof}
    Prior to beginning the proof we recall the characterization of parking functions from the introduction: a tuple $\beta$ is a parking function if and only if its rearrangement  into weakly increasing order, denoted $\beta^\uparrow=(b_1',b_2',\ldots,b_n')$, satisfies  $b_i'\leq i$ for all $i\in[n]$. 
    We will establish this inequality condition for the constructed tuple $\alpha$. 
    By construction, for each $i\in[n]$, the entry $a_i$ in $\alpha$ is assigned at most the parking spot where car $i$ ends up parking. Namely,
    \begin{align}
        a_i\leq \pi_i^{-1}\label{eq:inq satisfied}
    \end{align} for all $i\in[n]$. Here $\pi^{-1}$ is a permutation of $[n]$ so its rearrangement into weakly increasing order is the identity permutation $12\cdots n$.
    Then rearranging the entries of $\alpha$ into weakly increasing order, and denoting this by $\alpha^\uparrow=(a_1',a_2',\ldots,a_n')$, together with \Cref{eq:inq satisfied} ensures that 
    $a_i'\leq i$ for all $i\in[n]$. Thus $\alpha$ is a parking function.
    \end{proof}

    \begin{claim} The set of lucky cars of $\alpha$ is precisely $I$ i.e., $\Lucky(\alpha)=I$.
    \end{claim}
    \begin{proof}
        The only preferences used in $\alpha$ are the values $\pi_i^{-1}$ for $i\in I$. Moreover, all of the lucky cars in $I$, get the first instance of those preference as any unlucky car is given the preference of the nearest lucky car parked to its left, which makes those cars unlucky. This implies that the set of lucky cars in $\alpha$ is precisely~$I$.
    \end{proof}
        
        \begin{claim} The outcome of $\alpha$ is $\pi$  i.e., $\out{}(\alpha)=\pi$.
        \end{claim}
            \begin{proof}
By the same argument as in \cite[Theorem 2.2]{MVP}, $\out{}((\pi^{-1}_1,\pi_2^{-1},\ldots,\pi_n^{-1}))=\pi$. As $\alpha$ is constructed from $\pi^{-1}$ an entry at a time, giving car $j=\pi_i$
preference $\pi_j^{-1}$ whenever it is lucky, and giving the unlucky cars the preference of the previous lucky car, the outcome of $\alpha$ is exactly the same as the  outcome of $\pi^{-1}$, when we consider $\pi^{-1}$ as a parking function with entries $(\pi_1^{-1},\pi_2^{-1},\ldots,\pi_n^{-1})$. Thus $\out(\alpha)=\pi$ as desired. 
\end{proof}

This completes the proof.
\end{proof}

\subsection{Counting outcomes  where the first $k$ cars are lucky}

In Theorem \ref{thm:characterizing outcomes} we gave a complete characterization for the permutations that arise as outcomes of parking functions with a fixed set of lucky cars. 
We now set some notation to describe those permutations.

\begin{definition}\label{def:outcomes}
    Fix a lucky set $I$ of $\PF_n$.
    The set of \textit{outcomes} for $I$ is the subset of $\Sym_n$ defined as 
\[\out_n(I)\coloneq\{\pi\in \Sym_n: \out(\alpha)=\pi\mbox{ for some }\alpha\in \LPF_{n}(I)\}.\]
\end{definition}

\begin{example}\label{ex: n=5 lucky cars 1 and 4}
    Let $n=5$ and $I=\{1,4\}$. 
    By Theorem \ref{thm:characterizing outcomes}, $\pi=\pi_1\pi_2\pi_3\pi_4\pi_5\in\out_4(\{1,4\})$ if and only if 
    \begin{enumerate}[leftmargin=.5in]
        \item $\pi_1=1$ or $\pi_1=4$
        \item if $\pi_i\in I$ and $\pi_{i+1}\notin I$, then $\pi_i<\pi_{i+1}$, and 
        \item if $\pi_i-1>\pi_i$, then $\pi_i\in I$.
    \end{enumerate}
    Condition (3) implies that $\DB(\pi)\subseteq I$, which we proved in \Cref{lem:descent bottoms are lucky}.
Let us construct these permutations one condition at a time. 
There are $2\cdot 4!=48$ permutations in $n=5$ satisfying Condition~(1), as $\pi=1\sigma$ with $\sigma \in \Sym_{\{2,3,4,5\}}$ or $\pi=4\sigma$ with $\sigma\in \Sym_{\{1,2,3,5\}}$.
By Condition (2): 
\begin{itemize}[leftmargin=.25in]
\item If $\pi=1\sigma$ with $\sigma=\sigma_1\sigma_2\sigma_3\sigma_4\in \Sym_{\{2,3,4,5\}}$, then  $2\leq \sigma_1\leq 5$. To ensure $\pi$ also satisfies Condition (3),  $\sigma$ can either have no descents, or a descent bottom at $4$. Hence $\pi$ can only be 
$12345$ or $12354$.
\item If $\pi=4\sigma$ with $\sigma=\sigma_1\sigma_2\sigma_3\sigma_4\in \Sym_{\{1,2,4,5\}}$, then Condition (3) implies that $\sigma$ must have only a descent bottom at $1$.
 Hence $\pi$ can only be $41235$ or $45123$.
\end{itemize}
Thus $\out_5(\{1,4\})=\{12345,12354,41235,45123\}$.
\end{example}

We now give a formula for the number of parking outcomes when the first $k$ cars are lucky.
\begin{theorem}\label{thm:number of outcomes with first k cars lucky}
     Let $I=\{1,2,3,\ldots,k\}\subseteq[n]$ be a lucky set of $\PF_n$. Then 
     \[|\out_n(I)|=\sum_{J=\{j_1=1,j_2,j_3,\ldots,j_k\}\subseteq[n]}k!\binom{n-k}{j_2-j_1-1,j_3-j_2-1,\ldots,j_{k}-j_{k-1}-1,n-j_{k}}.\]
 \end{theorem}
\begin{proof}
  Let $I=\{1,2,3,\ldots,k\}$ be the set of lucky cars. 
  Fix a set of indices of parking spots $J=\{j_1=1,j_2,\ldots,j_k\}\subseteq[n]$, satisfying $j_1<j_2<\ldots<j_k\leq n$, in which the $k$ lucky cars park. 
  The lucky cars can park in any order among the spots in $J$, so there are $k!$ many ways for those cars to appear in an outcome permutation with lucky set $I$. 
  Since the lucky cars park on the street first, we need to park the $n-k$ unlucky cars, which are numbered $k+1,k+2,\ldots,n$. 
  The unlucky cars must park in the gaps between any two consecutively parked lucky cars. 
  The gaps between any two consecutively parked cars have size $j_{i+1}-j_{i}-1$ whenever $1\leq i\leq k-1$, and (at the end of the street) the gap from the last parked lucky car in position $j_k$ and the end of the street has length $n-j_k$.
  We can select the unlucky cars that park in each respective gap in $\binom{n-k}{j_2-j_1-1,j_3-j_2-1,\ldots,j_{k}-j_{k-1}-1,n-j_{k}}$ ways. 
  For each such selection of cars parking in those gaps, the unlucky cars must park in increasing order, thereby ensuring that those cars create no descents in the outcome, as that would imply that an unlucky car is lucky, which would give a contradiction. 
  
  Taking the sum over all possible subsets $J$ containing $1$ and of size $k$, we find that the number of outcomes of parking functions with lucky cars in the set $I=\{1,2,3,\ldots,k\}$ is given by 
\[|\out_n(I)|=\sum_{J=\{j_1=1,j_2,j_3,\ldots,j_k\}\subseteq[n]}k!\binom{n-k}{j_2-j_1-1,j_3-j_2-1,\ldots,j_{k}-j_{k-1}-1,n-j_{k}},\]
  as claimed.
\end{proof}
We conclude by noting that although we have a complete characterization for the elements of $\out_n(I)$ for any lucky set $I$, giving a closed formula for the cardinality of the set in general, remains an open problem.

\subsection{Counting parking functions with a fixed set of lucky cars}\label{sec:pfs with fixed lucky set}

Throughout this section $I$ is a lucky set of $\PF_n$. 
Hence, by Theorem~\ref{thm:lucky sets}, $1\in I$. We begin with the following definitions.

\begin{definition}\label{def:number of cars arriving before me parking to the left of me}
    For $\pi=\pi_1\pi_2\cdots\pi_n\in \Sym_n$ 
    let 
    $\ell(\pi_i)$
    be the length of the  longest subsequence $\pi_j\pi_{j+1}\cdots\pi_{i-1}$ with $\pi_t<\pi_i$ for all $j\leq t\leq i-1$.
\end{definition}
\begin{example}\label{ex:lengths for a particular pi}
If $\pi=195348267\in \Sym_9$, then 
$\ell(\pi_1=1)=0$,
$\ell(\pi_2=9)=1$,
$\ell(\pi_3=5)=0$,
$\ell(\pi_4=3)=0$,
$\ell(\pi_5=4)=1$,
$\ell(\pi_6=8)=3$,
$\ell(\pi_7=2)=0$,
$\ell(\pi_8=6)=1$,
$\ell(\pi_9=7)=2$.
\end{example}

For each $i\in[n]$, the definition of $\ell(\pi_i)$ counts the number of cars arriving before car $\pi_i$ that parked contiguously and  immediately to the left of car $\pi_i$. We show (in Lemma \ref{lem:possible prefs given an outcome}) that in the case that car $\pi_i$ is an unlucky car, the number $\ell(\pi_i)$ is exactly the number of potential spots car $\pi_i$ could prefer and which car $\pi_i$ would find occupied by cars that arrived and parked before it. 
We make this precise next.

\begin{definition}\label{def:set of possible outcomes}
    Fix a lucky set $I$ of $\PF_n$ and a permutation $\pi=\pi_1\pi_2\cdots\pi_n\in\out_n(I)$.
    For each $i\in[n]$, let $\Pref_I(\pi_i)$ denote the set of preferences of car $\pi_i$ in a parking function $\alpha$ of length $n$ satisfying $\Lucky (\alpha)=I$ and $\out(\alpha)=\pi$.
\end{definition}

\begin{example}[Continuing Example \ref{ex:lengths for a particular pi}] \label{ex:unlucky preferences is the lentgh of subsequence}If $\pi=195348267\in \Sym_9$ and car $\pi_8=6$ is unlucky, then it could only prefer the spot occupied by car $\pi_7=2$ which parked immediately to its left. 
This agrees with $\ell(\pi_8=6)=1$. 
If car $\pi_6=8$ is unlucky then it could prefer the spots occupied by cars $\pi_{3}=5$, $\pi_{4}=3$, and $\pi_{5}=4$, as those cars are parked immediately to the left of car $\pi_6=8$, and arrived prior to car $\pi_6=8$. So car $\pi_6=8$ has 3 preferences which would make it an unlucky car. This agrees with $\ell(\pi_6=8)=3$.
\end{example}

The count for the preferences of unlucky cars illustrated in Example \ref{ex:unlucky preferences is the lentgh of subsequence} holds in general. We prove this next.

\begin{lemma}\label{lem:possible prefs given an outcome}
    Fix a lucky set $I$ of $\PF_n$ and fix $\pi=\pi_1\pi_2\cdots\pi_n\in\out_n(I)$. Then
    \[|\Pref_I(\pi_i)| 
    =\begin{cases}
    1&\mbox{if $\pi_i\in I$}\\
    \ell(\pi_i)&\mbox{if $\pi_i\notin I$}
    \end{cases}\]
    and the number of possible parking functions $\alpha$ with outcome $\pi$ and lucky set $I$ is equal to
\[\prod_{i=1}^n |\Pref_{I}(\pi_i)|=\prod_{\pi_j\notin I}\ell(\pi_j).\]
\end{lemma}
\begin{proof}
    Let $\alpha=(a_1,a_2,\ldots,a_n)\in\PF_n$ satisfying 
    $\Lucky(\alpha)=I$ and $\out(\alpha)=\pi=\pi_1\pi_2\cdots\pi_n\in\Sym_n$. For each $i\in[n]$, car $\pi_i$ is either lucky or unlucky. 
    If car $\pi_i$ is lucky, then car $\pi_i$ parking in spot $i$ must have had spot $i$ as its preference. This is a unique preference, which agrees with the value of $\ell(\pi_i)=1$ for all $\pi_i\in I$.
    
    Now consider the case where $i\in[n]$ for which car $\pi_i$ is an unlucky car that parks in spot $i$. Hence, $\pi_i\notin I$.
    In this case, the only possible preferences that car $\pi_i$ can have, so as to park in spot $i$ and be unlucky, is the set of spots occupied immediately to the left of spot $i$ provided those spots were occupied by cars that arrived and parked before car $\pi_i$. 
    That is, if the sequence of cars  $\pi_j\pi_{j+1}\cdots\pi_{i-1}$ is the longest sequence satisfying $\pi_t<\pi_i$ for all $j\leq t\leq i-1$, then the unlucky car $\pi_i$ can prefer any of the spots numbered $j,j+1,\ldots,i-1$. This is precisely the value of $\ell(\pi_i)$ when $\pi_i\notin I$.

As the car preferences are independent of each other, we have that the product of $|\Pref_I(\pi_i)|$ over all $i\in[n]$ counts  all of the possible preferences for cars $1$ to $n$ so that they have outcome $\pi$ and lucky set $I$. Thus, the number of possible parking functions $\alpha$ with outcome $\pi$ and lucky set $I$ is equal to
\[\prod_{i=1}^n |\Pref_{I}(\pi_i)|=\left(\prod_{\pi_i\in I} \ell(\pi_i)\right)\cdot \left(\prod_{\pi_j\notin I}\ell(\pi_j)\right)=\prod_{\pi_j\notin I}\ell(\pi_j),\]
where the last equality holds as $\ell(\pi_i)=1$ for all $\pi_i\in I$.
\end{proof}

\begin{example}[Continuing Example \ref{ex:lengths for a particular pi}]
Observe that $\pi=195348267\in \Sym_9$ has $\DB(\pi)=\{1=\pi_1,2=\pi_7,3=\pi_4,5=\pi_3\}$, and so $\pi$ can be the outcome of a parking function with lucky set $I$ provided $\{1=\pi_1,2=\pi_7,3=\pi_4,5=\pi_3\}\subseteq I$.
Suppose that $I=\{1,2,3,5,7\}$. Then, by Lemma \ref{lem:possible prefs given an outcome}, the number of parking functions $\alpha$ satisfying $\out(\alpha)=\pi$ and $\Lucky(\alpha)=I$ is given by 
\[\prod_{i=1}^9|\Pref_I(\pi_i)|=\prod_{\pi_j\notin I}\ell(\pi_j)=
\ell(\pi_{5}=4)\cdot \ell(\pi_{8}=6)\cdot \ell(\pi_{6}=8)\cdot
\ell(\pi_{2}=9)=1\cdot 1\cdot 3\cdot 1=3.\]
In fact, the only parking functions $\alpha\in\PF_9$ with lucky set $I$ and outcome $\pi$ are:
\[(1,7,4,4,3,7,9,\fbox{3},1),(1,7,4,4,3,7,9,\fbox{4},1),\mbox{    and   }(1,7,4,4,3,7,9,\fbox{5},1),\]
where we box the difference between these parking functions,
which denotes  the possible preferences of car $\pi_6=8$ being spots $3$, $4$, and $5$, as those spots are occupied by earlier cars $\pi_3=5$, $\pi_{4}=3$, and $\pi_{5}=4$, respectively. Hence, car $8$ is forced to park in spot $6$.

If instead we consider the lucky set $J=\{1,2,3,5,8\}$, then, by Lemma \ref{lem:possible prefs given an outcome}, the number of parking functions $\alpha$ satisfying $\out(\alpha)=\pi$ and $\Lucky(\alpha)=J$ is 
\[\prod_{i=1}^9|\Pref_{J}(\pi_i)|=\prod_{\pi_j\notin J}\ell(\pi_j)=
\ell(\pi_{5}=4)\cdot \ell(\pi_{8}=6)\cdot \ell(\pi_{9}=7)\cdot
\ell(\pi_{2}=9)=1\cdot 1\cdot 2\cdot 1=2.\]
Hence, then parking functions $\alpha\in\PF_9$ with lucky set $J=\{1,2,3,5,8\}$ and outcome $\pi$ are:
\[(1,7,4,4,3,7,\fbox{7},6,1),\mbox{    and   }(1,7,4,4,3,7,\fbox{8},6,1),\]
where we box the difference between these parking functions, which denotes  the possible preferences of car $7$ being spots $7$ and $8$, as those spots are occupied by cars $2$ and $6$, respectively. Hence, car $7$ is forced to park in spot $9$.
\end{example}

We now give a formula for the number of parking functions with a fixed set of lucky cars. Our proof technique is similar to that of Colmenarejo et al., as
presented in \cite[Corollary 3.1]{colmenarejo2021counting}.

\begin{theorem}\label{thm:count pfs with fixed lucky set}
Fix a lucky set $I$ of $\PF_n$. Let $\out_n(I)$ denote the set of outcome permutations for parking functions with lucky set $I$.
Then the number of parking functions with lucky set $I$ is
\begin{align}
|\LPF_n(I)|=
\sum_{\pi=\pi_1\pi_2\cdots\pi_n\in \out_n(I)}
\left(\prod_{\pi_j\notin I} \ell (\pi_j)\right).
\end{align}
\end{theorem}
\begin{proof}
Fix a lucky set $I\subseteq[n]$ of $\PF_n$.
By Lemma \ref{lem:possible prefs given an outcome}, for any $\pi=\pi_1\pi_2\cdots\pi_n\in \out_n(I)$, the number of parking functions $\alpha\in \LPF_n(I)$ with outcome $\pi$ is given by $\prod_{\pi_j\notin I}\ell(\pi_j)$.
To account for all parking functions with lucky set $I$, we must sum over all permutations  $\pi$ arising 
as the outcomes such parking functions. 
Those are the permutations characterized in 
Theorem \ref{thm:characterizing outcomes} and that set is denoted by $\out_n(I)$.
This completes the proof.
\end{proof}

\begin{example}\label{ex: continuing n=5 lucky cars 1 and 4}
In \Cref{ex: n=5 lucky cars 1 and 4} we determined that the set of outcomes of parking functions of length $n=5$ with lucky set $I=\{1,4\}$ is  $\out_n(I)=\{12345,12354,41235,45123\}$.
Next, for each $\pi\in \out_5(I)$, we use \Cref{lem:possible prefs given an outcome} to compute the number of parking functions with lucky set $I$ and outcome $\pi$: 
\begin{itemize}[leftmargin=.2in]
    \item If $\pi=12345$, then there are $
    \displaystyle\prod_{\pi_j\in\{2,3,5\}}\ell(\pi_j)=1\cdot2 \cdot4=8$ parking functions.
    \item
    If $\pi=12354$, then there are $\displaystyle\prod_{j\in\{2,3,5\}}\ell(\pi_j)=1\cdot 2\cdot3=6$ parking functions.
    \item
    If $\pi=41235$, then there are $\displaystyle\prod_{\pi_j\in\{2,3,5\}}\ell(\pi_j)=1\cdot2\cdot4=8$ parking functions.
    
    \item
    If $\pi=45123$, then 
    there are $\displaystyle\prod_{\pi_j\in\{2,3,5\}}\ell(\pi_j)=1 \cdot2\cdot1=2$ parking functions.
    
\end{itemize}
Thus, by Theorem \ref{thm:count pfs with fixed lucky set}, the number of parking functions of length $5$ with lucky set $I=\{1,4\}$ is given by 
\begin{align}
    |\LPF_5(\{1,4\})|&=\sum_{\pi\in\out_5(\{1,4\})}\left(\prod_{\pi_j\in \{2,3,5\}}\ell(\pi_j)\right)=8+6+8+2=24.\notag
\end{align}
\end{example}

\subsection{Weakly increasing parking functions with a fixed set of lucky cars}

In this section, we let $\PF_{n}^{\uparrow}$ denote the set of parking functions $\alpha=(a_1,a_2,\ldots,a_n)\in \PF_n$ satisfying $a_i\leq a_{i+1}$ for all $i\in[n-1]$.
The number of elements in $\PF_n^{\uparrow}$ is given by the $n$th Catalan number \[\Cat_n=\frac{1}{n+1}\binom{2n}{n},\] which is OEIS sequence \cite[\seqnum{A000108}]{OEIS}.

We begin by noting that the proof of Theorem~\ref{thm:lucky sets} directly establishes that any subset of $[n]$ including $1$ is a lucky set for $\PF_n^{\uparrow}$. We state this formally next.

\begin{lemma}
    A subset $I\subseteq[n]$ is a lucky subset of $\PF_n^{\uparrow}$ if and only if $1\in I$.
\end{lemma}

For a fixed lucky set $I$, let $\LPF_n^{\uparrow}(I)$ denote the set of weakly increasing parking functions $\alpha\in\PF_n^{\uparrow}$  with $\Lucky(\alpha)=I$. We now give a formula for the number of weakly increasing parking functions with a fixed set of lucky cars. 

\begin{theorem}\label{thm:weakly increasing pfs with a fixed set of lucky cars}
Fix a lucky set $I=\{i_1,i_2,i_3,\ldots,i_k\}\subseteq[n]$ of $\PF_n^{\uparrow}$ and assume $1=i_1<i_2<\cdots<i_k\leq n$.
Then 
\begin{align}\label{eq:catalan product}
|\LPF_n^{\uparrow}(I)|=\prod_{j=1}^k\Cat_{x_j},
\end{align}    
where 
$x_{j}=i_{j+1}-i_j-1$ for each $j\in [k-1]$ and $x_k=n-i_k$.
\end{theorem}
\begin{proof}
    If $\alpha\in\PF_n^{\uparrow}$, then the cars park in order and, hence, $\out(\alpha)=\pi=123\cdots n\in\Sym_n$. 
    Namely, for all $i\in [n]$, car $i$ parks in spot $i$.
    Thus, regardless of lucky set $I$, the outcome of a weakly increasing parking function $\alpha$ has no descents nor descent bottoms except at $\pi_1$.
    Thus, if $\alpha=(a_1,a_2,\ldots,a_n)\in\PF_n^{\uparrow}$ has lucky set $I$, then for each $i_j\in I$, with $1\leq j\leq k$, car  $i_j$  has preference $a_{i_j}=i_j$.
    While, for any $j\in[k-1]$, the cars $i_{j}+1, i_{j}+2,\ldots, i_{j+1}-1$ must have weakly increasing preferences 
    $a_{i_{j}+1}\leq 
a_{i_{j}+2}\leq\cdots\leq a_{i_{j+1}-1}$, respectively, with
    $i_j\leq a_{x}\leq x-1$ for all $x\in[i_{j}+1,i_{j+1}-1]$.
    The number of list satisfying this condition is precisely the number of weakly increasing parking functions of length $i_{j+1}-1-(i_{j}+1)+1=i_{j+1}-i_{j}-1$. 
    The number of which is given by the Catalan number $\Cat_{i_{j+1}-i_{j}-1}$. 
    In a similar way, the preferences for cars $i_k+1,i_k+2,\ldots, n$ must also be weakly increasing and satisfy $i_k\leq a_x\leq x-1$ for all $x\in[i_k+1, n]$.
    The number of which is given by the Catalan number $\Cat_{n-i_k}$.
    The result follows from taking the product of the resulting Catalan numbers, which gives the formula in Equation~\eqref{eq:catalan product}.
\end{proof}
\begin{example}[Continuing Example \ref{ex: continuing n=5 lucky cars 1 and 4}]
Let $n=5$ and $I=\{1,4\}$.
Let $i_1=1$ and $i_2=4$. Then 
using the definition of $x_j$ as given in \Cref{thm:weakly increasing pfs with a fixed set of lucky cars}, observe that 
$x_1=4-1-1=2$ and $x_2=5-4=1$.
Then, by \Cref{thm:weakly increasing pfs with a fixed set of lucky cars}, the number of weakly increasing parking functions of length $5$ with lucky set $\{1,4\}$ is $\Cat_{2}\cdot \Cat_{1}=2\cdot 1=2$. 
One can readily confirm that the only weakly increasing parking functions of length $5$ with lucky set $I=\{1,4\}$  are $(1,1,1,4,4)$ and $(1,1,2,4,4)$.
\end{example}

\section{More spots than cars}\label{sec:more spots than cars}
    One common generalization of parking functions is the case where there are more parking spots than cars. 
    We let $\PF_{m,n}$ denote the set of $(m,n)$-parking functions with $m$ cars and $n$ spots, and we assume that $1\leq m\leq n$.
    Then the set $\PF_{m,n}$ consists of all $m$-tuples $\beta=(b_1,b_2,\dots, b_m)\in [n]^m$ whose weakly increasing rearrangement denoted $\beta^\uparrow=(b_1',b_2',\ldots, b_m')$ satisfies \begin{align}\label{inequality description of mn pfs}
        b_i'\leq n-m+i\mbox{ for all $i\in[m]$.}
    \end{align} 
 Konheim and Weiss established that $|\PF_{m,n}|=(n+1-m)(n+1)^{m-1}$, see\cite[Lemma 2]{konheim1966occupancy}. 
    
    We now consider the analogous question of how many parking functions in $\PF_{m,n}$ have a fixed set of lucky cars. In this case, the analysis only differs in that lucky cars have more options among the spots in which they might park. 
    Also among the possible outcomes, the only change is made by first  selecting an $m$-element subset of $[n]$ to index the location of where the $m$ cars park on the street, as the remaining spots would be empty. So the proof of some results are analogous to the case when $m=n$, in which case we omit their details, but we give their formal statements below. Whenever the argument differs substantially we provide a detailed proof.

This section is organized analogously to the case where $m=n$. That is, we begin by describing all of the possible lucky sets. We then  characterize the types of permutations that are outcomes of parking functions with a fixed set of lucky cars, and then we give a formula for the number of parking functions with a fixed set of lucky cars.

\subsection{Lucky sets for $(m,n)$-parking functions}
As before, given $\alpha\in\PF_{m,n}$, we let $\Lucky(\alpha)$ denote the set of lucky cars of $\alpha$. 
Given a subset $I\subset[m]$, we let \[\LPF_{m,n}(I)=\{\alpha\in\PF_{m,n}:\Lucky(\alpha)=I\}.\] 
Whenever $I\subseteq[m]$ is such that $\LPF_{m,n}(I)\neq\emptyset$, we say that $I$ is a lucky set of $\PF_{m,n}$.
We begin by characterizing and enumerating the lucky sets of $\PF_{m,n}$.

\begin{theorem}\label{thm:lucky sets of mn}
A subset $I\subseteq[m]$ is a lucky set of $\PF_{m,n}$ if and only if $1\in I$.
\end{theorem}

\begin{corollary}\label{cor:number lucky sets of mn}
For $m\geq 1$, there are $2^{m-1}$ distinct lucky sets of $\PF_{m,n}$.    
\end{corollary}

\begin{corollary}
    If $I\subseteq[m]$ is a lucky set of $\PF_{m,n}$, then it is a lucky set of $\PF_{k,\ell}$ for all $k\geq m$ and $\ell \geq n$ with $\ell\geq k$.
\end{corollary}

\subsection{Outcomes of $(m,n)$-parking functions with a fixed lucky set}\label{sec:outcomes in mn}

Next we define the outcome of a parking function when there are more parking spots than cars.

\begin{definition}\label{def:outcome for mn}
Let $\Sym_{m,n}$ denote the set of permutations of the multiset $\{\X,\ldots,\X\}\cup[m]$ with $\X$ having multiplicity $n-m$. Given a parking function $\alpha=(a_1,a_2,\ldots,a_m) \in \PF_{m,n}$, define the \textit{outcome} of $\alpha$ by \[\mathcal{O}(\alpha)=\pi_1\pi_2\cdots\pi_n\in \Sym_{m,n}\] where $\pi_i=j\in[m]$ denotes that car $j$ parked in spot $i$, and $\pi_i=\X$ indicates that spot $i$ remained vacant.
\end{definition}

For sake of simplicity in our arguments, we assume that $\X>i$ for any $i\in \mathbb{N}$. 
Moreover, to every element $\pi=\pi_1\pi_2\cdots\pi_n\in\Sym_{m,n}$ we will prepend $\pi_0=\X$.
We now give our definition for descents and descent bottoms for permutations in $\Sym_{m,n}$.
\begin{definition}\label{def:descent bottoms in mn}
Let $\pi_0=\X$ and $\pi_1\pi_2\cdots\pi_n\in\Sym_{m,n}$. 
Let $\pi=\pi_0\pi_1\pi_2\cdots\pi_n=\X\pi_1\pi_2\cdots\pi_n$.
\textit{Descent bottoms} of $\pi$ are values $\pi_i\in[m]$ for which $\pi_{i-1}=\X$ or,  $\pi_{i-1}\in[m]$ and $\pi_{i-1}>\pi_i$.
In these cases, $i$ is called a \textit{descent} of $\pi$.
The set
    \begin{align}
    \Des_{m,n}(\pi)&\coloneq\{i\in [n]: i \mbox{ is a descent of }\pi\in\Sym_{m,n}\},
    \intertext{
    is called the \textit{descent set} of $\pi$, and }
    \DB_{m,n}(\pi)&\coloneq\{\pi_i\in [n]: \pi_i \mbox{ is a descent bottom of }\pi\in\Sym_{m,n}\},\nonumber
\end{align}
is called the \textit{descent bottom set} of $\pi$.
\end{definition}
We illustrate the above definitions with the following example.
\begin{example}
Let $\alpha\in \PF_{5,7}$ and
$\pi=\pi_0\pi_1\pi_2\pi_3\pi_4\pi_5\pi_6\pi_7=\X453\X1\X2$. The descent bottoms of $\pi$ are
$\pi_1=4$, $\pi_3=3$, $\pi_5=1$, and $\pi_7=2$.
All of these cars must be in the lucky set of $\alpha$, as those cars would have found the spot immediately to their left vacant upon their attempt to park, either because that spot remained empty after all cars parked, or because the car that ultimately parked there had not yet entered the street to park. 
Hence, if $\out(\alpha)$ is to be equal to $\pi$, then $\{1,2,3,4\}\subseteq \Lucky(\alpha)$.
\end{example}

By using Definition \ref{def:descent bottoms in mn}, we can extend Lemma \ref{lem:descent bottoms are lucky} to the case where there are more spots than cars.

\begin{lemma}\label{lem:descent bottoms are lucky in mn}
Fix $\pi\in \Sym_{m,n}$. For any $\alpha\in \PF_{m,n}$ with outcome $\out(\alpha)=\pi$, with $\pi\in\Sym_{m,n}$, if $i\in \Des(\X\pi)$, then  $\pi_i\in \Lucky(\alpha)$. In other words,
$\DB_{m,n}(\pi)\subseteq\Lucky(\alpha)$.
\end{lemma}
\begin{proof}
The proof of Lemma \ref{lem:descent bottoms are lucky in mn} is analogous to the proof of Lemma \ref{lem:descent bottoms are lucky}, as it is immaterial to car $\pi_i$ whether $\pi_{i-1}=\X$ or $\pi_{i-1}\in[m]$ since in either case car $\pi_i$ would find spot $i-1$ empty, and hence would park in spot $i$ and be a lucky car.
\end{proof}

As before, Lemma \ref{lem:descent bottoms are lucky in mn} establishes if $\alpha$ has outcome $\pi$, then all descent bottoms of $\pi$ are lucky cars of $\alpha$. 
We also provide an analogous result to Theorem \ref{thm:characterizing outcomes} to characterize which permutations in $\Sym_{m,n}$ are the outcomes of a parking function in $\PF_{m,n}$ with a fixed set of lucky cars.

\begin{theorem}\label{thm:characterizing outcomes in mn}
Fix a lucky set $I\subseteq[m]$  of $\PF_{m,n}$. Then $\pi=\pi_1\pi_2\cdots\pi_n\in \Sym_{m,n}$ is the outcome of $\alpha\in\LPF_n(I)$
if and only if $\X\pi=\X\pi_1\pi_2\ldots\pi_n$ satisfies the following:
\begin{enumerate}[leftmargin=.5in]
    \item if $\pi_{i-1}=\X$ and $\pi_i\in[m]$, then $\pi_i\in I$, 
    \item if $\pi_i\in I$ and $ \pi_{i+1} \notin I\cup\{\X\}$, then $\pi_i<\pi_{i+1}$,
    \item if 
    $\pi_{i-1},\pi_{i}\in[m]$ and 
    $\pi_{i-1}>\pi_{i}$, then $\pi_{i}\in I$. 
\end{enumerate}
\end{theorem}

\begin{proof}
Before we proceed with the technical details we make the following remarks regarding the items in the theorem statement:
\begin{enumerate}[leftmargin=.5in]
    \item The statement ``If $\pi_{i-1}=\X$ and $\pi_i\in[m]$, then $\pi_i\in I$'' means that a car $\pi_i$ parking after an empty spot must be lucky.
    \item The statement ``If $\pi_i\in I$ and $ \pi_{i+1} \notin I\cup\{\X\}$, then $\pi_i<\pi_{i+1}$'' means that anytime we have two consecutively parked cars, if the first  car is lucky and second is not, then there is an ascent in the outcome.
    \item The statement ``If 
    $\pi_{i-1},\pi_{i}\in[m]$ and 
    $\pi_{i-1}>\pi_{i}$, then $\pi_{i}\in I$'' means that if there are two consecutively parked cars  and there is a descent in the outcome, then second car is lucky. 
\end{enumerate}
We use these facts and proceed with the technical details next.\\

\noindent($\Rightarrow$) 
Suppose that $\pi=\pi_1\pi_2\cdots\pi_n=\out(\alpha)$ and that the lucky cars of $\alpha$  are those indexed by the set $I$.
Any car parking after an empty spot is lucky, so it must be in $I$, which is precisely Condition (1). Also whenever there are two consecutively parked cars $\pi_{i}$ and $\pi_{i+1}$, if $\pi_{i}$ is lucky and the second is not lucky, then car $\pi_{i+1}$ had to be in the queue later, and hence $\pi_{i}<\pi_{i+1}$. This is precisely the statement of Condition (2).
Moreover, whenever two consecutively parked cars $\pi_{i-1}$ and $\pi_i$ satisfy $\pi_{i-1}>\pi_i$, then car $\pi_i$ arrived and parked prior to car $\pi_{i-1}$ entering and parking, so it must be a lucky car. This is precisely Condition (3). 
So $\pi$ satisfies the given conditions.\\

\noindent ($\Leftarrow$) Assume $\pi=\pi_1\pi_2\cdots\pi_m\in\Sym_{m,n}$ and $\X\pi$ satisfies Conditions (1), (2), and (3).
We use the following notation:
\begin{itemize}[leftmargin=.5in]
    \item For all $i\in [n]$, $\pi_i=j$ means car $j$ parked in spot $i$, and if $\pi_i=\X$, then no car parks in spot $i$.
    \item Let $J=\{i\in [n]: \pi_i \in [m]\}.$ Further, order the elements of $J$ from smallest to largest so that $J=\{j_1<j_2<\cdots<j_m\}$. 
    For each $i\in[m]$, let $\tau_i=\pi_{j_i}$ and define
    \begin{align}
\tau=\tau_1\tau_2\cdots\tau_m.\label{eq:tau}
    \end{align}
    \item For any $u\in J$, define $\pi_v^{-1}=u$ which denotes that car $v$ parks in spot $u$ on the street.
\end{itemize} 
Next, we will construct a parking function $\alpha=(a_1,a_2,\ldots,a_m)\in \PF_{m,n}$, and we will prove that $\alpha$ satisfies $\out(\alpha)=\pi$ and $\Lucky(\alpha)=I$. 

To construct $\alpha$
let $\tau$ be as defined in \Cref{eq:tau}. For each $\tau_i=1,2,\ldots, m$ (in this order) set the value of $a_{j}$ where $j=\tau_i$ as follows:

\begin{enumerate}[leftmargin=.5in]
    \item If $\tau_i=j\in I$, set $a_{j}=\pi_{j}^{-1}$.\label{step:1}
    \item If $\tau_i=j\notin I$, set $a_{j}$ as follows:\label{step:2}
    \begin{enumerate}
        \item If $\tau_{i-1}\in I$ and $\tau_{i-1}<\tau_i$, then set $a_{j}=\pi_{k}^{-1}$ where $k=\tau_{i-1}$.
        \item If $\tau_{i-1}\in I$ and $\tau_{i-1}>\tau_i$, then in fact $\tau_{i}$ would be in $I$, a contradiction. Thus this case never happens.
        \item If $\tau_{i-1}\notin I$, then set $a_{j}=\pi_{k}^{-1}$ where $k=\tau_{i-1}$.
    \end{enumerate}
\end{enumerate}
We must now establish the following claims:
\setcounter{claim}{0}
    \begin{claim} As constructed, $\alpha$  is a parking function in $\PF_{m,n}$.
    \end{claim}
     \begin{proof}
We now show that the weakly increasing rearrangement of $\alpha$ satisfies  the inequality in \Cref{inequality description of mn pfs}.
    By construction, 
    for each $i\in[m]$, the entry $a_i$ in $\alpha$ is assigned at most the parking spot where car $i$ ends up parking. Namely, 
    \begin{align}    a_i\leq\pi_{i}^{-1}\label{eq:inq satisfied mn}
     \end{align}
     for all $i\in[m]$.
    Thus,
    the entries in $\alpha=(a_1,a_2,\ldots,a_m)$ are
    entry-wise less than or equal to the entries in $(\pi_{1}^{-1},\pi_{2}^{-1},\ldots,\pi_{m}^{-1})$. 
By definition of $J$, the set $\{\pi_{1}^{-1},\pi_{2}^{-1},\ldots,\pi_{m}^{-1}\}=J$, where $J$ consisted of the indices at which the cars parked given the outcome $\pi$.
    Now, reorder the entries $(\pi_{1}^{-1},\pi_{2}^{-1},\ldots\pi_{m}^{-1})$ in weakly increasing order and denote this by $(\pi_{1}^{-1},\pi_{2}^{-1},\ldots,\pi_{m}^{-1})^\uparrow$. 
    Again by definition of $J$, we have that 
    \[(\pi_{1}^{-1},\pi_{2}^{-1},\ldots,\pi_{m}^{-1})^\uparrow=j_1j_2j_3\cdots j_m.\]
    We now consider $(\pi_{1}^{-1},\pi_{2}^{-1},\ldots,\pi_{m}^{-1})^\uparrow$ as a parking function. We will prove that the $i$th entry in the tuple $(\pi_{1}^{-1},\pi_{2}^{-1},\ldots,\pi_{m}^{-1})^\uparrow$ is less than or equal to $n-m+i$, for all $i\in[m]$.
    This is equivalent to the claim that   
    $j_i\leq n-m+i$, for each $i\in[m]$. 
    Assume for contradiction that $ j_i\geq n-m+i+1$. 
    Then car $i$ parks at a  spot numbered $m-n+i+1$ or larger. 
    After car $i$ parks, there are at most $n-(n-m+i+2)+1=m-i-1$ spots remaining at the end of the street, yet there are $m-i$ cars left to park.
    Hence, some car would be unable to park.
    This is a contradiction because we started with $J$ being the set of indices at which cars parked with outcome $\pi$.
    
    Hence $j_i\leq n-m+i$ for all $i\in[m]$. 
    Then rearranging the entries of $\alpha$ into weakly increasing order, and denoting this by $\alpha^\uparrow=(a_1',a_2',\ldots,a_m')$, together with \Cref{eq:inq satisfied mn} ensures that 
    $a_i'\leq m-n+i$ for all $i\in[m]$. Thus $\alpha$ is a parking function in $\PF_{m,n}$.
    \end{proof}
    
    \begin{claim} The set of lucky cars of $\alpha$ is precisely $I$ i.e., $\Lucky(\alpha)=I$.
    \end{claim}
    \begin{proof}
        The only preferences used to create $\alpha$ are the values $\pi_{\tau_i}^{-1}$ when $\tau_i\in I$. 
        Moreover, all of the lucky cars in $I$, get the first instance of those preference as any unlucky car is given the preference of the nearest lucky car parked to its left, which makes those cars unlucky. This implies that the set of lucky cars in $\alpha$ is precisely~$I$.
    \end{proof}
    \begin{claim}The outcome of $\alpha$ is $\pi$  i.e., $\out{}(\alpha)=\pi$.
    \end{claim}
    \begin{proof}
Consider the first car parked on the street, by \Cref{eq:tau} this is denoted  $\tau_1=\pi_{j_1}$.
By definition of $J$, car $\tau_1=\pi_{j_1}$ occupies spot $j_1$ on the street.
Since the first parked car on the street is always lucky, the preference for car $\tau_1=\pi_{j_1}$ was set to be $a_{\tau_1}=\pi_{\tau_1}^{-1}$ in Step \eqref{step:1} above.
As car $\tau_1=\pi_{j_1}$ is lucky, then it must park in its preference $\pi_{\tau_1}^{-1}$. 
Whenever we write $\pi_x^{-1}=y$, it means that car $x$ parks in spot $y$. Thus $\pi_{\tau_1}^{-1}=j_1$
means that car $\tau_1=\pi_{j_1}$ parks in spot $j_1$.
This implies that $\pi_{j_1}=\tau_1$, as this denotes that car $\tau_1$ parked in spot $j_1$.
Note that this implies that $\pi_0=\pi_1=\pi_2=\cdots=\pi_{\tau_1-1}=\X$.

Assume for induction that for all $1<\ell <m$, the car $\tau_\ell$ (the $\ell$th car parked on the street),
with preference $a_{\tau_\ell}$ satisfies $\pi_{j_\ell}=\tau_\ell$.

Now consider the $m$th car. This is car $\tau_m$. 
The preference for car $\tau_m$ depends on whether the car was lucky or unlucky. We consider each case next.
\begin{itemize}[leftmargin=.2in]
    \item If car $\tau_m$ is lucky, then in Step \eqref{step:1} we set the preference $a_{\tau_m}=\pi_{\tau_m}^{-1}$. 
    Since car $\tau_m$ is the $m$th car, it has to park in the $m$th occupied spot on the street, which by definition of $J$, is numbered $j_m$.
    So $\pi_{\tau_m}^{-1}=j_m$, since whenever we write $\pi_x^{-1}=y$ means that car $x$ parked in spot $y$.
    Since car $\tau_m$ parked in spot $j_m$, we have that $\pi_{j_m}=\tau_{m}$, as desired.

    \item If car $\tau_m$ is unlucky, then in Step \eqref{step:2} its preference was set to be $a_{\tau_m}=\pi_{\tau_{m-1}}^{-1}$.
By induction hypothesis car $\tau_{m-1}$ satisfies $\pi_{j_{m-1}}=\tau_{m-1}$. Meaning that car $\tau_{m-1}$ parked in the spot numbered $j_{m-1}$. 
Car $\tau_m$ wants to park in spot $\pi_{\tau_{m-1}}^{-1}$. 
However, $\pi_{j_{m-1}}=\tau_{m-1}$ implies that 
$\pi_{\tau_{m-1}}^{-1}=j_{m-1}$.
This tells us that car $\tau_{m}$ will find its preferred spot occupied by car $\tau_{m-1}$. 
Thus, car $\tau_{m}$ will proceed to the next available spot, which by definition of $J$ must be $j_m$.
Therefore $\pi_{j_m}=\tau_m$. 
\end{itemize}
Thus all of the $m$ cars with preferences in $\alpha$ have outcome agreeing with the non-$\X$ entries in $\pi$.
Whenever $j\notin J$, then $\pi_j=\X$. 
Therefore $\out(\alpha)=\X\pi$ as claimed.
\end{proof}
\noindent This completes the proof.
\end{proof}

\subsection{Counting outcomes of $(m,n)$-parking functions where the first $k$ cars are lucky}

In Theorem \ref{thm:characterizing outcomes in mn} we gave a complete characterization for the permutations that arise as outcomes of parking functions with a fixed set of lucky cars. 
We now set some notation to describe those permutations.

\begin{definition}\label{def:set of possible outcomes for mn}

    Fix a lucky set $I$ of $\PF_{m,n}$.
    Let $\out_{m,n}(I)$ denote the set of outcomes arising from the parking functions in $\PF_{m,n}$ with lucky set $I$.
\end{definition}

We now give a formula for the number of outcomes where the first $k$ cars are lucky. This result generalizes \Cref{thm:number of outcomes with first k cars lucky} to the case where there are more spots than cars.
\begin{theorem}\label{thm:number of outcomes with first k cars lucky in mn}
     Let $I=\{1,2,3,\ldots,k\}\subseteq[m]$ be a lucky set of $\PF_{m,n}$, and let $J=\{j_1,j_2,j_3,\ldots,j_k\}\subseteq[n]$ with $1\leq j_1\leq n-m+1$.
Let $\mathrm{WComp}(m-k,k)$ denote the set of weak compositions of $m-k$ with $k$ parts.
Let $\mathrm{WComp}_J(m-k,k)$ be the subset of $\mathrm{WComp}(m-k,k)$ for which 
     $(t_1,t_2,\dots,t_k)\in \mathrm{WComp}(m-k,k)$ satisfies
     $(t_1,t_2,\dots,t_k)\leq (j_2-j_1+1,j_3-j_2+1,\ldots,n-j_k)$ point-wise.
     Then 
     \[|\out_{m,n}(I)|=\sum_{(t_1,t_2,t_3,\ldots,t_k)\in \mathrm{WComp}_J(m-k,k)}k!\binom{m-k}{t_1,t_2,t_3,\ldots,t_k}.\]
 \end{theorem}
\begin{proof}
  Let $I=\{1,2,3,\ldots,k\}\subseteq[m]$ and fix a set of indices of parking spots $J=\{j_1,j_2,\ldots,j_k\}\subseteq[n]$, satisfying $j_1<j_2<\cdots<j_k\leq n$, in which the $k$ lucky cars in $I$ park. 
  The lucky cars in $I$ can park in any order among the spots indexed by $J$, so there are $k!$ ways for those cars to appear in an outcome permutation. 
  As those cars all park on the street first, we need to now park all of the remaining unlucky cars, of which there are $m-k$. These unlucky cars must park in the gaps between any two consecutively parked lucky cars. If $j_1>1$, then there is at least one empty spot at the beginning of the street. Any unlucky car must park to the right of the first lucky car on the street. Namely, any unlucky car must park past spot $j_1$, as otherwise the unlucky car would be lucky. 
  Moreover, we claim that $j_1\leq n-m+1$, which means that the first lucky car can park in any of the first $n-m+1$ spots on the street. Assume for contradiction that $j_1>n-m+1$.
  Then observe that there are $m-k$ (unlucky) cars left to park after the $k$ lucky cars park. 
  Now, the number of spots to the right of $j_1$ is $n-j_1+1-k$ and $n-j_1+1-k\leq n-(n-m+2)+1-k=m-k-1$. 
  This means that there are fewer spots to the right of $j_1$ than cars left to park. This is a contradiction because all unlucky cars must park to the right of any lucky car, and if $j_1>n-m+1 $ some unlucky car would fail to park. 
  Thus, $1\leq j_1 \leq n-m+1$, as claimed. 
  
  Next, the gaps between any two consecutively parked cars have size $j_{i+1}-j_{i}-1$ whenever $1\leq i\leq k-1$, and the gap at the end of the street (from the last parked lucky car in position $j_k$ to the end of the street) has length $n-j_k$.
  We must select what cars park in each respective gap. 
  Whenever we select cars to park in those gaps, the unlucky cars must park immediately to the right of a lucky car and must appear in increasing order.
  Moreover, the gaps are limited in size and may contain up to $ j_{i+1}-j_{i}+1$ cars where $1\leq i\leq k-1$, while the gap at the end of the street may contain up to $n-j_k$ cars.
  The number of cars we must park in these gaps must equal the number of unlucky cars, which is $m-k$. 
  As some gaps may not have any unlucky cars parked in them, selecting the number of cars to park in the gaps under these constraints can be accounted by $(t_1,t_2,\ldots,t_k)$, a weak composition of $m-k$ with $k$ parts which satisfies $(t_1,t_2,\dots,t_k)\leq (j_2-j_1+1,j_3-j_2+1,\ldots,n-j_k)$ point-wise. 
  Then, given the composition $(t_1,t_2,\dots,t_k)$,  
  the number of ways to select what cars park in each respective gap is given by $\binom{m-k}{t_1,t_2,\ldots,t_k}$.
  Each selection of cars parking in those gaps, must have the selected cars park in increasing order so as to ensure they create no descents in the outcome, which would imply some car not $I$ is lucky.
  
  Taking the sum over all possible restricted weak compositions in $\mathrm{WComp}_J(m-k,k)$, we find that the number of outcomes with lucky set $I=\{1,2,3,\ldots,k\}$ is given by 
\[|\out_{m,n}(I)|=\sum_{(t_1,t_2,t_3,\ldots,t_k)\in \mathrm{WComp}_J(m-k,k)}k!\binom{m-k}{t_1,t_2,t_3,\ldots,t_k},\]
  as claimed.
\end{proof}

We again remark that although we have a complete characterization for the elements of $\out_{m,n}(I)$ for any lucky set $I$, giving a closed formula for the cardinality of the set in general, remains an open problem.

\subsection{Counting $(m,n)$-parking functions with a fixed set of lucky cars}
If $\pi=\pi_1\pi_2\cdots\pi_{n}\in\Sym_{m,n}$ is an outcome of an $(m,n)$-parking function and car $\pi_i$ is an unlucky car, then there must be a car parked in spot $i-1$, namely $\pi_{i-1}\in[m]$. Moreover, car $\pi_{i-1}$ had to have arrived earlier than car $\pi_i$ to park. So $\pi_{i-1}<\pi_i$. Otherwise, if $\pi_{i-1}$ was an empty spot, then parking in spot $i$ would have only occurred if car $\pi_i$ had spot $i$ as its preference. This would make car $\pi_i$ a lucky car, contrary to the initial assumption.
Thus the count for the possible preferences of an unlucky car $\pi_i$ depends on the number of cars that parked contiguously to the left of spot $i$ by cars arriving to park before car $\pi_i$ entered the street to park. We define this number next.

\begin{definition}\label{def:number of cars arriving before me parking to the left of me mn version}
     Fix $\pi=\pi_1\pi_2\cdots\pi_n\in \Sym_{m,n}$. 
    For all $\pi_i\in[m]$ let 
    $\ell(\pi_i)$
    be the length of the  longest subsequence $\pi_j\pi_{j+1}\cdots\pi_{i-1}$ with $\pi_t<\pi_i$ for all $j\leq t\leq i-1$.
\end{definition}
We stress that $\ell(\pi_i)$ is only defined when $\pi_i\in[m]$ and not when $\pi_i=\X$, as the latter indicates an empty spot and not a car.

\begin{definition}
        Let $\pi=\pi_1\pi_2\cdots\pi_n\in\out_{m,n}(I)$.
    For each $i\in[n]$, 
    if $\pi_i\in[m]$, then 
    let $\Pref_I(\pi_i)$ denote the set of preferences of car $\pi_i$ in $\alpha\in\PF_{m,n}$ satisfying $\Lucky (\alpha)=I$ and whose outcome is $\pi$.
    Moreover, 
we let
\[\out^{-1}_{m,n}(I,\pi)=\{\alpha\in\PF_{m,n}:\Lucky(\alpha)=I\mbox{ and }\out(\alpha)=\pi\} 
    \]
    be the set of $(m,n)$-parking functions with lucky set $I$ and outcome $\pi$.
\end{definition}

We now state an analogous result to Lemma \ref{lem:possible prefs given an outcome} for counting $(m,n)$-parking functions with a fixed lucky set and a given outcome. 
As the proof is analogous to that of  \Cref{lem:possible prefs given an outcome}, hence we omit the details.

\begin{lemma}\label{lem:possible prefs given an outcome - mn version}
    Fix a lucky set $I\subseteq[m]$ of $\PF_{m,n}$ and fix $\pi=\pi_1\pi_2\cdots\pi_n\in\out_{m,n}(I)$. Then for each $\pi_i\in[m]$,
    \[|\Pref_I(\pi_i)| 
    =\begin{cases}
    1&\mbox{if $\pi_i\in I$}\\
    \ell(\pi_i)&\mbox{if $\pi_i\notin I$}
    \end{cases}\]
    and the number of possible $(m,n)$-parking functions $\alpha$ with outcome $\pi$ and lucky set $I$ is equal to
\[|\out^{-1}_{m,n}(I,\pi)|=\prod_{i=1}^n |\Pref_{I}(\pi_i)|=\prod_{\pi_j\notin I\cup\{\X\}}\ell(\pi_j).\]
\end{lemma}

We now give a formula for the number of $(m,n)$-parking functions with a fixed set of lucky cars.

\begin{theorem}\label{thm:count pfs with fixed lucky set - mn version}
Fix a lucky set $I\subseteq[m]$ of $\PF_{m,n}$.
If $\LPF_{m,n}(I)$ denotes the set of $(m,n)$-parking functions $\alpha$ with $\Lucky(\alpha)=I$, then 
\begin{align}
|\LPF_{m,n}(I)|=
\sum_{\pi=\pi_1\pi_2\cdots\pi_n\in \out_{m,n}(I)}
\left(\prod_{\pi_j\notin I\cup\{\X\}} \ell (\pi_j)\right).\label{eq:pf count with fixed lucky set - mn version}
\end{align}
\end{theorem}
\begin{proof}
Fix a lucky set $I\subseteq[m]$  of $\PF_{m,n}$.
By Lemma \ref{lem:possible prefs given an outcome - mn version}, for any $\pi=\pi_1\pi_2\cdots\pi_n\in \out_{m,n}(I)$, the number of $(m,n)$-parking functions $\alpha\in \LPF_{m,n}(I)$ with outcome $\pi$ is given by $\prod_{\pi_j\notin I\cup\{\X\}}\ell(\pi_j)$.
To account for all parking functions with lucky set $I$, we must sum over all permutations  $\pi$ arising 
as the outcomes of such parking functions. 
Those are the permutations characterized in 
Theorem \ref{thm:characterizing outcomes in mn} and that set is denoted by $\out_{m,n}(I)$.
This completes the proof.
\end{proof}

\subsection{Weakly increasing $(m,n)$-parking functions with a fixed set of lucky cars}

In this section, we let $\PF_{m,n}^{\uparrow}$ denote the set of parking functions $\alpha=(a_1,a_2,\ldots,a_m)\in \PF_{m,n}$ which are weakly increasing i.e., $a_i\leq a_{i+1}$ for all $i\in[m-1]$.
Theorem~\ref{thm:lucky sets of mn} establishes that any subset of $[m]$ including $1$ is a lucky set for $\PF_{m,n}^{\uparrow}$. We state this formally next.
\begin{lemma}
    A subset $I\subseteq[m]$ is a lucky set of $\PF_{m,n}^{\uparrow}$ if and only if $1\in I$.
\end{lemma}

For a fixed lucky set $I$, let \[\LPF_{m,n}^{\uparrow}(I)=\{\alpha\in\PF_{m,n}^{\uparrow}:\Lucky(\alpha)=I\}, \]
which is the set of all weakly increasing $(m,n)$-parking functions with lucky set $I$.
Let 
\[\out_{m,n}^\uparrow(I)=\{\pi\in \Sym_{m,n}: \out(\alpha)=\pi\mbox{ for some }\alpha\in \LPF_{m,n}^{\uparrow}(I)\}\] denote the set of outcomes arising from the parking functions in $\PF_{m,n}^\uparrow$ with lucky set $I$.

We begin with an example where we construct the possible outcomes of a weakly increasing $(m,n)$-parking function with a fixed set of lucky cars.
\begin{example}\label{ex:outcomes mn}
    Let $m=7$ and $n=10$, and $I=\{1,4,5\}$.
    Since $\alpha\in\PF_{7,10}^\uparrow$, the non-$\X$ entries in $\out(\alpha)$ will appear in increasing order. 
    So we may begin with the identity permutation in $\Sym_7$ and introduce three $\X$'s ensuring that the $\X$'s are only placed to the left of a lucky car  or at the end of the street. This placement of $\X$'s ensures that no other car is forced to be lucky. 
    Hence, the only outcomes that are possible are
    \[
    \out_{m,n}^\uparrow(I)=\left\{
    \begin{tabular}{ccccc}
         $\X\X\X1234567$,&
         $\X\X123\X4567$,&
         $\X\X1234\X567$,&
         $\X\X1234567\X$,&
         $\X123\X\X4567$,\\
         $\X123\X4\X567$,&
         $\X123\X4567\X$,&
         $\X1234\X\X567$,&
         $\X1234\X567\X$,&
         $\X1234567\X\X$,\\
         $123\X\X\X4567$,&
         $123\X\X4\X567$,&
         $123\X\X4567\X$,&
         $123\X4\X\X567$,&
         $123\X4\X567\X$,\\
         $123\X4567\X\X$,&
         $1234\X\X\X567$,&
         $1234\X\X567\X$,&
         $1234\X567\X\X$,&
         $1234567\X\X\X$.
         \end{tabular}
    \right\}.
    \]
    To count these outcomes it suffices to know the number of weak compositions of $n-m=3$ into $|I|+1=4$ parts. This is given by $\binom{n-m+|I|}{n-m}=\binom{6}{3}=20$, which is the number of outcomes listed above.
\end{example}

\Cref{ex:outcomes mn} illustrates the proof technique we use in establishing the following result.
\begin{lemma}\label{lem:number of outcomes in weakly inc pfs mn}

        Fix a lucky set $I\subseteq[m]$ of $\PF_{m,n}^\uparrow$.
    Then the number of permutations in $\Sym_{m,n}$ that are outcomes of parking functions in $\PF_{m,n}^\uparrow$ with lucky set $I$ is given by 
\[|\out_{m,n}^\uparrow(I)|=\binom{n-m+|I|}{n-m}.\]
\end{lemma}
\begin{proof}
    Let $\alpha\in\PF_{m,n}^{\uparrow}$ and let $\out(\alpha)=\pi_1\pi_2\cdots\pi_n\in\Sym_{m,n}$.  
    As $\alpha$ is weakly increasing, the entries in $[m]$ appear in strictly increasing order in $\out(\alpha)$.
    By Theorem \ref{thm:characterizing outcomes in mn}, entries in $\out(\alpha)$ which are $\X$ impact the lucky set. So we may insert $\X$ into the identity permutation $123\cdots m\in\Sym_m$ only immediately to the left of a lucky car or to the right of $m$. 
    As there are $|I|+1$ places where we can insert $n-m$ entries $\X$ in $123\cdots m\in\Sym_m$, this is equivalent to counting the number of weak compositions of $n-m$ into $|I|+1$ parts. It is well-known that the number of weak compositions of $n-m$ into $|I|+1$ parts is given by $\binom{n-m+|I|+1-1}{n-m}=\binom{n-m+|I|}{n-m}$, as claimed.
\end{proof}

\Cref{lem:number of outcomes in weakly inc pfs mn} gives the number of permutations that are outcomes of weakly increasing $(m,n)$-parking functions with a fixed set of lucky cars. We now want to characterize those $(m,n)$-parking functions directly and to do so we introduce the following definitions and results.

\begin{definition}\label{def:weak composition}
    Let $\pi=\pi_1\pi_2\cdots\pi_n\in\Sym_{m,n}$ be an outcome of a weakly increasing parking function in $\PF_{m,n}^\uparrow$ with lucky set $I$. Hence $\pi_j=\X$ for  $n-m$ distinct values of $j\in[n]$. Partition $\pi$ (from left-to-right) into subwords $w_1,w_2,w_3,\ldots,w_k$ where either $w_i=\X$ or $w_i$ consists of a maximally long increasing run of positive integers from the set $[m]$. 
    Whenever $w_i=\X$, let $c_i=0$, and for each $i\in[k]$ with $w_i\neq \X$, let $c_i$ denote the length of the word $w_i$.
    We call 
     $\textbf{c}(\pi)=(c_1,c_2,\ldots,c_k)$ the \textit{weak composition associated to $\pi$}. 
\end{definition}

In \Cref{def:weak composition}, we use the naming convention of ``weak composition'' since for any $\textbf{c}(\pi)=(c_1,c_2,\ldots,c_k)$, we have that 
 $\sum_{i=1}^kc_i=m$ and $n-m$ of the terms in the sum are zero. Thus $\textbf{c}(\pi)$ is a weak composition of~$m$. We illustrate the definition with the next example.

\begin{example}\label{ex:pi to comp and back}
Let $m=7$, $n=10$, and $I=\{1,4,5\}$.
Consider $\pi=\X123\X4\X567\in\Sym_{7,10}$. Then 
the weak composition associated to $\pi$ is given by $\textbf{c}(\pi)=(0,3,0,1,0,3)$.
If we had been given the weak composition $(0,3,0,1,0,3)$, we can recover $\pi$
as follows: 
\begin{itemize}[leftmargin=.2in]
    \item Since  $c_1=0$, we let $w_1=\X$.
    \item Since $c_2=3$, we use the smallest $c_2=3$ integers in $[7]$ and place them in increasing order to construct $w_2=123$.
    \item Since $c_3=0$, we let $w_3=\X$.
    \item Since $c_4=1$,  we then use the next $c_4=1$ smallest integers in $\{4,5,6,7\}=[7]\setminus[\sum_{i=1}^{3}c_i]$ to construct $w_4=4$. 
    \item Since $c_5=0$, we let $w_5=\X$.
    \item Since $c_6=3$, we use the next smallest $c_6=3$ integers in $\{5,6,7\}=[7]\setminus[\sum_{i=1}^{5}c_i]$ to construct $w_6=567$.
\end{itemize}
Then $\pi=w_1w_2\cdots w_6=\X123\X4\X567$, which is the permutation we started with.
\end{example}

From \Cref{def:weak composition}, given an outcome $\pi$ we can uniquely determine the weak composition associated to~$\pi$.
As \Cref{ex:pi to comp and back} illustrates, the converse is also true and we prove this next.

\begin{lemma}\label{lem:wc to perm}
    Let $1\leq m\leq n$. A weak composition $\textbf{c}=(c_1,c_2,\ldots,c_k)$ of $m$ with $n-m$ zeros, uniquely determines an outcome permutation $\pi\in\Sym_{m,n}$ of a weakly increasing $(m,n)$-parking function. 
\end{lemma}
\begin{proof}
    Consider the  weak composition $\textbf{c}=(c_1,c_2,\ldots,c_k)$ of $m$ with $n-m$ zeros. 
    We now construct a set of subwords $w_1,w_2,\ldots,w_k$ from the entries in $\textbf{c}$ and we use those subwords from which we will construct a permutation in $\Sym_{m,n}$. This construction is as follows: For each $i\in[k]$,
\begin{itemize}[leftmargin=.65in]
    \item[Step (1):]  If $c_i=0$,  then set $w_i=\X$.
    \item[Step (2):] If  $c_i>0$, then set
    \[w_{i}=\left(\left(\sum_{j=1}^{i-1}c_{j}\right)+1\right)\left(\left(\sum_{j=1}^{i-1}c_{j}\right)+2\right)\cdots\left(\left(\sum_{j=1}^{i-1}c_{j}\right)+c_i\right),\] which consists of the consecutive integers $\left(\sum_{j=1}^{i-1}c_{j}\right)+y$, with $1\leq y\leq c_i$ listed in increasing order. Moreover,  
    if the top index of the sum is larger than the initial index, then the sum is zero.
\end{itemize}
Then let $\pi=w_1w_2\cdots w_k$.
By Step (1) there are $n-m$ entries in $\pi$ which are $\X$ and by Step (2) the numbers in $[m]$ all appear in $\pi$ exactly once and in increasing order from left to right. This shows that $\pi$ is an outcome of a weakly increasing $(m,n)$-parking function with $m$ cars and $n$ spots, as claimed.
The fact that $\pi$ is unique follows from the fact that the construction of the words $w_1,w_2,\ldots,w_k$ are uniquely determined by the entries $c_1,c_2,\ldots,c_k$ of \textbf{c}, respectively.
\end{proof}

   Knowing the outcome permutation (or equivalently a weak composition) is not enough information to know the lucky set of an $(m,n)$-parking function with that outcome (or equivalently weak composition). This is analogous to the fact that descent bottoms of a permutation in $\Sym_n$ only gave a subset of the lucky cars of a parking function with a given outcome, a result we established in \Cref{lem:descent bottoms are lucky}. 
   The main issue is that the lucky set places a restriction on the positive parts appearing in the weak composition associated to an outcome.

For a fixed lucky set $I$, we will determine the set of weak compositions that correspond to outcomes of elements in $\LPF_{m,n}^\uparrow(I)$.
Before presenting this result we give a few definitions and technical results.

\begin{definition}\label{def:partial sums}
    Let $\textbf{c}$ be the associated weak composition of $\pi$. Let $\textbf{p}$ denote the sequence of positive integers in $\textbf{c}$, appearing in the same order as in $\textbf{c}$, and let $|\textbf{p}|$ denote the number of parts of $\textbf{p}$. We call $\textbf{p}$ the \textit{composition associated to $\pi$}.

    Given a composition $\textbf{p}=(p_1,p_2,\ldots,p_k)$ define its \textit{partial sums list} by \[\textbf{s}=(s_0,s_1,s_2,\ldots,s_k),\]
    where $s_0=0$ and $s_i=\sum_{j=1}^ip_j$ for each $1\leq i\leq k$.
\end{definition}
We illustrate \Cref{def:partial sums} with the next example.
\begin{example}\label{ex:previous}
    Let $\pi=1234\X56789\X10\X(11)(12)\X\X$.

The weak composition associated to $\pi$ is given by $\textbf{c}=(4,0,5,0,1,0,2,0,0)$ and the composition associated to $\pi$ is $\textbf{p}=(4,5,1,2)$. The partial sum list is $\textbf{s}=(0,4,9,10,12)$.
\end{example}

Not every composition is the composition associated to an outcome of a weakly increasing $(m,n)$-parking function with a given lucky set. For example, the only  composition associated to an outcome of a weakly increasing parking function with lucky set $I=\{1\}$ is $\textbf{p}=(m)$, because if $\textbf{p}=(a,b)$ with $a,b>0$ and $a+b=m$, then the weak composition associated to $\textbf{p}$ would have at least one zero between the parts $a$ and $b$ in $\textbf{p}$. However, such a zero would indicate that in $\pi$, the outcome permutation, entry $\pi_{a+1}=\X$, implying that car $a+1$ would be lucky, a contradiction since we assumed that the only lucky car was the first car as $I=\{1\}$. 

To determine all of the possible compositions corresponding to an outcome of an element in $\LPF_{m,n}^\uparrow(I)$, we must describe how the entries in the lucky set $I$ restrict the types of parts appearing in the associated composition. 
The main idea will be to begin with a tuple $\textbf{u}$ associated to the lucky set $I$ and use it to construct a set of tuples $\{\textbf{v}\}$ by replacing consecutive entries in $\textbf{u}$ with their sum in $\textbf{v}$, ensuring that the result is a weak composition that comes from an outcome of a weakly increasing $(m,n)$-parking function with lucky set $I$.

Before giving the technical definition we illustrate it with an example. 
\begin{example} \label{ex:insert bars}  
Let $m=22$ and $n=27$.  Consider \[\textbf{u}=(1,3,1,4,2,5,3,2,1),\] which is a partition of $m=22$. We will place $n-m=5$ bars in the tuple $\textbf{u}$. One way to do this is $(\vert1\vert3,1,4,2\vert5,3\vert2,1\vert)$. With such a placement of bars we then add the entries between bars to get $(\vert1\vert10\vert8\vert3\vert)$. Then we replace the bars with ``,'' to get the tuple $\textbf{v}=(1,10,8,3)$. Of course a distinct placement of bars would yield a different result. 
If instead we had $(\vert\vert1,3,1\vert4,2\vert5,3,2\vert1)$, then $\textbf{v}=(5,6,10,1)$. We also could insert all bars at the end or start of $\textbf{u}$, in which case  $\textbf{v}=(22)$.
\end{example}
We now give the formal definition of the construction illustrated in \Cref{ex:insert bars}.

\begin{definition}\label{def: good comps}
    Let $1\leq m\leq n$ and let $\textbf{u}=(u_1,u_2,\ldots,u_k)$ be a composition of $m$. 
    Insert $n-m$ bars into $\textbf{u}$, denoted ``$\vert$''.
    Compute the sums of any integers between pairs of bars, or a parenthesis and a bar (at the start/end of the tuple). 
    Then delete any
    bars immediately following the opening parenthesis ``$($'' or immediately preceeding the closing parenthesis ``$)$'', and replace all other bars with commas.

    Let $\mathcal{S}(\textbf{u})$ denote the set all of possible tuples constructed in this way, which we refer to as the \textit{good compositions of} $\textbf{u}$. 
    Whenever we need to specify the number of parts, we let $\mathcal{S}_x(\textbf{u})$ denote the set of good compositions of $\textbf{u}$ with exactly $x+1$ parts.
\end{definition}

The set of  good compositions of a specific tuple related to the set of lucky cars will help in determining the parts appearing in the compositions of weakly increasing $(m,n)$-parking functions with a fixed lucky set. We illustrate this next before proceeding to give a more technical result.
\begin{example}\label{ex:computing good comps}
    Let $m=9$, $n=11$, and $I=\{1,4,5\}\subseteq[9]$. Let $\textbf{u}$ be the tuple constructed from the consecutive differences of the increasing entries in $I$ with a final entry appended which ensures $\textbf{u}$ is a composition of $m=9$. Hence, $\textbf{u}=(4-1,5-4,9+1-5)=(3,1,5)$. 
    Let $k=n-m=11-9=2$, then the good compositions of $\textbf{u}$ are the elements in the set
    \begin{align}\label{eq:good comps of 315}\mathcal{S}((3,1,5))&=\mathcal{S}_0((3,1,5))\cup\mathcal{S}_1((3,1,5))\cup\mathcal{S}_2((3,1,5))
    \intertext{where}\nonumber
        \mathcal{S}_0((3,1,5))&=\{(9)\}\\\nonumber
        \mathcal{S}_1((3,1,5))&=\{(4,5),(3,6)\},\text{ and }\\\nonumber
        \mathcal{S}_2((3,1,5))&=\{(3,1,5)\}.
    \end{align}
We now show that each of the compositions of $9$ in  $\mathcal{S}((3,1,5))$ corresponds to a composition associated to an outcome of a weakly increasing $(9,11)$-parking function with lucky set $I$. 
By \Cref{lem:number of outcomes in weakly inc pfs mn}, there should be $\binom{11-9+3}{11-9}=\binom{5}{2}=10$ outcomes.
Using the same technique as in \Cref{ex:outcomes mn} we construct all of the outcome permutations in $\Sym_{9,11}$ with lucky set $I=\{1,4,5\}$ and compute their associated composition:
\begin{align}\label{eq:outcomes and comps for 315}
\begin{tabular}{ll}
     \X\X123456789, \X123456789\X, 123456789\X\X, & \text{ which have associated composition $(9)$,} \\
     \X123\X456789, 123\X\X456789, 123\X456789\X,& \text{ which have associated composition $(3,6)$,}\\
     \X1234\X56789, 1234\X\X56789, 1234\X56789\X,& \text{ which have associated composition $(4,5)$, and }\\
     123\X4\X56789,& \text{ which have associated composition $(3,1,5)$.}\\
\end{tabular}\end{align}
Confirming that the good compositions of $(3,1,5)$ are precisely the compositions associated to outcomes of weakly increasing $(9,11)$-parking functions with lucky set $I=\{1,4,5\}$.

However, if  instead $n=10$, then, as $n-m=1$, we would only be able to place a single bar in $(3,1,5)$. 
So any good composition associated to an outcome would have at most $n-m+1=2$ parts.
Hence, $\mathcal{S}((3,1,5))=\{(9),(4,5),(3,6)\}$.  
\end{example}
\Cref{ex:computing good comps} illustrates a general fact: good compositions of $\textbf{u}$ constructed with at most $n-m$ bars, will have at most $n-m+1$ parts.
Our next result gives a collection of 
good compositions for a tuple related to the lucky set.

\begin{proposition}\label{prop:allowable comps}
    Let $1\leq m\leq n$ and fix a lucky set $I=\{1=i_1<i_2<\cdots<i_k\}\subseteq[m]$ of $\PF_{m,n}$. 
    Let $\textbf{i}=(i_2-i_1,i_3-i_2,\ldots,i_{k-1}-i_{k-2},i_{k}-i_{k-1},m+1-i_k)$. 
    If  $\textbf{p}$ is a good composition of $\textbf{i}$, namely $\textbf{p}\in \mathcal{S}(\textbf{i})$, with at most $n-m+1$ parts,
    then there exists an outcome permutation whose associated composition is $\textbf{p}$.
\end{proposition}
\begin{proof}
    Let $\textbf{p}=(p_1,p_2,\ldots,p_k)$ be a good composition of $\textbf{i}$ with $k\leq n-m+1$. 
    Then, by definition, 
    $\sum_{i=1}^{k}p_i=m$.
    Since $n-m$ is the number of empty spots on the street, then any outcome of a weakly increasing $(m,n)$-parking function with lucky set $I$ will include $n-m$ entries equal to $\X$.
    As $\textbf{p}$ has $k$ entries, with $k\leq n-m+1$, then there are $n-m$ zeros available which we can use to construct a weak composition by inserting a $0$ in $\textbf{p}$ between every two positive values (there are $k-1$ such places to insert a zero) and 
    we append the remaining $n-m-(k-1)$ zeros to the end of the weak composition, which we 
    denote by
    \[\textbf{c}=(p_1,0,p_2,0,\ldots,p_{k-1},0,p_k,\underbrace{0,0,\ldots,0}_{n-m-(k-1)\text{ terms}})=(c_1,c_2,\ldots,c_{n-m+k}).\]
    Using the algorithm in the proof of \Cref{lem:wc to perm}, we construct 
    \[\small \pi=\underbrace{123\cdots p_1}_{\text{$p_1$ numbers}}~\X~\underbrace{(p_1+1)(p_1+2)\cdots\left(\sum_{x=1}^2p_x\right)}_{\text{$p_2$ numbers}}~\X\cdots\X~\underbrace{\left(\left(\sum_{x=1}^{k-1}p_x\right)+1\right)\left(\left(\sum_{x=1}^{k-1}p_x\right)+2\right)\cdots \left(\sum_{x=1}^{k}p_x\right)}_{\text{$p_k$ numbers}}~\underbrace{\X\cdots\X}_{\begin{matrix}{\scriptstyle{n-m-(k-1)}}\\ \tiny{\text{terms}}\end{matrix}}.\]
Since
    there are $n-m-(k-1)+(k-1)=n-m$ entries in $\pi$ which are $\X$ and since all of the numbers in $[m]$ appear in $\pi$ exactly once and in increasing order from left to right, we know that $\pi$ is an outcome of a weakly increasing $(m,n)$-parking function, as claimed.
    Lastly, by \Cref{def:weak composition}, the weak composition associated to $\pi$ is exactly $\textbf{c}$, and the composition associated to $\pi$ is exactly $\textbf{p}$, which concludes the proof.
\end{proof}

We now describe how to generate all outcome permutations (equivalently weak compositions) from the set of good compositions.

\begin{proposition}\label{prop:description of all outcomes}
     Let $1\leq m\leq n$ and fix a lucky set $I=\{1=i_1<i_2<\cdots<i_k\}\subseteq[m]$ of $\PF_{m,n}$. 
    Let $\textbf{i}=(i_2-i_1,i_3-i_2,\ldots,i_{k-1}-i_{k-2},i_{k}-i_{k-1},m+1-i_k)$ and let $\textbf{p}$ be a good composition of $\textbf{i}$.
    Then the following statements hold:
    \begin{enumerate}[leftmargin=.5in]
        \item  The good composition $\textbf{p}$ of $\textbf{i}$ is a composition associated to an outcome of a weakly increasing $(m,n)$-parking function with lucky set $I$ if and only if $\textbf{p}$ has $x+1$ parts, where $0\leq x\leq n-m$.
        \item There are \[\binom{|I|-1}{x}\] good compositions of $\textbf{i}$ with $x+1$ parts, where $0\leq x\leq n-m$.
        \item For each good composition $\textbf{p}$ of $\textbf{i}$ with $x+1$ parts, where $0\leq x\leq n-m$, there are \[\binom{n-m+1}{n-m-x}\]
        weak compositions that are outcomes of a weakly increasing $(m,n)$-parking function with lucky set~$I$. 
        \item The set of weak compositions constructed from the set of good compositions of $\textbf{i}$ with $x+1$ parts, for all $0\leq x\leq n-m$, is in bijection with the set $\out_{m,n}^\uparrow(I)$, which is the set of outcomes of weakly increasing $(m,n)$-parking functions  with lucky set $I$.
    \end{enumerate}
    
\end{proposition}

\begin{proof}
    Proof of Statement (1). $(\Rightarrow)$ By \Cref{def: good comps}, the good compositions $\textbf{p}$ of 
    $\textbf{i}$ are constructed by placing $n-m$ bars in $\textbf{i}$. 
    These bars can be placed at the start or end of $\textbf{i}$, or between any two positive integers in $\textbf{i}$. If $x$ bars are placed in between two positive values in $\textbf{i}$, then the good composition constructed  has exactly $x+1$ parts. Thus, the most parts any good composition of $\textbf{i}$ can have is $x+1$ where $0\leq x\leq n-m$, as claimed.
    
    $(\Leftarrow)$ This direction holds by the construction in \Cref{prop:allowable comps}.
    \\

\noindent    Proof of Statement (2). A good composition $\textbf{p}$ of $\textbf{i}$ is uniquely determined by where we place the bars in $\textbf{i}$. As there are $|I|$ entries in $\textbf{i}$, then there are $|I|-1$ places where we can insert bars so as to affect the number of parts. By (1) we can insert $x$ bars, where $0\leq x\leq n-m$, among $|I|-1$ locations between two consecutive values in $\textbf{i}$. Hence, the number of good compositions $\textbf{p}$ of $\textbf{i}$ with $x+1$  parts, where $0\leq x\leq n-m$, is given by $\binom{|I|-1}{x}$, and as usual $\binom{a}{b}=0$, whenever $b>a$.\\

\noindent    Proof of Statement  (3). Given $\textbf{p}$, a good composition of $\textbf{i}$ with $x+1$ parts, where $0\leq x\leq n-m$, by \Cref{prop:allowable comps}, we must use $x$ of the $n-m$ zeros to place between the entries of $\textbf{p}$ as we construct the weak compositions with $n-m$ zeros, which correspond to outcomes of weakly increasing parking functions with lucky set $I$. This leaves $x+2$ positions at which we can insert the remaining $n-m-x$ zeros i.e., between two positive values or at the start or end of $\textbf{p}$. 
The number of ways we can do this is equal to the number of weak compositions of $n-m-k$ into $x+2$ parts, which is given by 
\[\binom{(n-m-x)+(x+2)-1}{n-m-x}=\binom{n-m+1}{n-m-x},\]
as claimed.\\

\noindent Proof of Statement  (4). By (1), (2), and (3) we have constructed 
\[\sum_{x=0}^{n-m}\binom{|I|-1}{x}\binom{n-m+1}{n-m-x}\]
weak compositions of $m$ with $n-m$ zeros. 
By \Cref{prop:allowable comps} we have confirmed that all of the constructed weak compositions are outcomes of weakly increasing $(m,n)$-parking functions with lucky set $I$. 
By \Cref{lem:number of outcomes in weakly inc pfs mn}, the number of outcomes of weakly increasing $(m,n)$-parking functions  with lucky set $I$ is given by $\binom{n-m+|I|}{n-m}$.
As the sets are finite, it suffices to show that 
\begin{align}\label{eq:looks like chu-vandermont identity}
\sum_{x=0}^{n-m}\binom{|I|-1}{x}\binom{n-m+1}{n-m-x}=\binom{n-m+|I|}{n-m}.
\end{align}
\Cref{eq:looks like chu-vandermont identity} follows directly from the Chu-Vandermont identity 
\begin{align}
    \sum_{k=0}^{\ell}\binom{r}{k}\binom{s}{\ell-k}=\binom{r+s}{\ell}
\end{align}
by letting $\ell=n-m$, $k=x$, $r=|I|-1$, and $s=n-m+1$.
\end{proof}

\begin{example}[Continuing \Cref{ex:computing good comps}]
    Let $m=9$, $n=11$, and $I=\{1,4,5\}\subseteq[9]$. 
    Then $\textbf{i}=(4-1,5-1,9+1-5)=(3,1,5)$. By Statement (2) in \Cref{prop:description of all outcomes}, we have $\binom{2}{0}=1$ good composition of $\textbf{i}$ with~$1$~part, $\binom{2}{1}=2$ good compositions of $\textbf{i}$ with $2$ parts, and $\binom{2}{2}=1$ good composition of $\textbf{i}$ with $3$ parts. This agrees with \Cref{eq:good comps of 315}. 
    By Statement (3) in \Cref{prop:description of all outcomes}, there are $\binom{3}{2}=3$ weak compositions associated to compositions of $\textbf{i}$ with $1$ part, $\binom{3}{1}=3$ weak compositions associated to compositions of $\textbf{i}$ with $2$ parts, and $\binom{3}{0}=1$ weak composition associated to compositions of $\textbf{i}$ with $3$ parts. This agrees with the data presented in \eqref{eq:outcomes and comps for 315}. 
    Lastly, by Statement (4) of \Cref{prop:description of all outcomes}, the number of outcomes of weakly increasing $(9,11)$-parking functions with lucky set $I=\{1,4,5\}$ is 
    \[\sum_{x=0}^{2}\binom{2}{x}\binom{3}{2-x}=\binom{2}{0}\binom{3}{2}+\binom{2}{1}\binom{3}{1}+\binom{2}{2}\binom{3}{0}=10,\]
    which agrees with the data presented in \eqref{eq:outcomes and comps for 315}.
\end{example}

The key take-away from \Cref{prop:allowable comps} is that, given a lucky set $I$, we can construct 
all of the outcomes of weakly increasing $(m,n)$-parking functions  with lucky set $I$ from the set of
good compositions of $\textbf{i}$ with at most $n-m+1$ parts. 
Next, we work towards enumerating the set of weakly increasing $(m,n)$-parking functions with a fixed outcome and a fixed set of lucky cars.

When $1\leq m\leq n$, an outcome permutation $\pi\in\Sym_{m,n}$ can be broken into subsets of the street on which the cars park, which we refer to as \textit{occupied sub-streets}. As defined, the composition associated to $\pi$ corresponds to the sequence of the lengths of the occupied sub-streets. 
Since the $(m,n)$-parking functions we are considering are weakly increasing, the cars all park on the street in order $1$ to $m$ in those occupied sub-streets. 
Moreover, the number of preferences of a car depends on whether the car is lucky or unlucky and also on its location within the occupied sub-street it parked in. 
When a car is lucky, and not the first car parked in an occupied sub-street, then we need to determine how far that car is from the start of the occupied sub-street in which it parked. The following function describes this distance and plays a key role in our count for the number of weakly increasing $(m,n)$-parking functions with a fixed outcome and a fixed set of lucky cars.

\begin{definition}\label{def:denoms}
         Let $1\leq m\leq n$ and fix the set of lucky cars $I=\{1=i_1<i_2<\cdots<i_k\}\subseteq[m]$. 
    Let $\textbf{i}=(i_2-i_1,i_3-i_2,\ldots,i_{k-1}-i_{k-2},i_{k}-i_{k-1},m+1-i_k)$ and let $\textbf{p}$ be a good composition of $\textbf{i}$. Let $\textbf{s}=(0,s_1,s_2,\ldots,s_k)$ be the partial sum list of $\textbf{p}$. 
    Define the function $d_{\textbf{p}}:I\to\mathbb{Z}^+$ by 
    \[d_{\textbf{p}}(i)=\begin{cases}
        1&\text{if $i=s_j+1$ for some $0\leq j\leq k$}\\
        i-(s_j+1)&\text{if $s_j+1<i<s_{j+1}$ for some $0\leq j\leq k-1$}\\
        m-(s_{k-1}+1)&\text{if $i=m$.}\\
    \end{cases}\]
\end{definition}
We illustrate \Cref{def:denoms} next.
\begin{example}[Continuing \Cref{ex:previous}]
Let $m=12$, $n=17$, and $I=\{1,3,5,8,10,11,12\}\subseteq[12]$.
Then $\textbf{i}=(2,2,3,2,1,1,1)$. 
Consider $\textbf{p}=(4,5,1,2)$, which is a good composition of $\textbf{i}$, whose partial sum list is $\textbf{s}=(0,4,9,10,12)$.
We can readily compute 
\begin{center}
\begin{tabular}{ll}
    $d_{\textbf{p}}(1)=1$ & since $1=1+s_0$,\\
    $d_{\textbf{p}}(3)=3-1=2$ &since $s_0+1=1<3<4=s_1$,\\
    $d_{\textbf{p}}(5)=1$ &since $s_1+1=5$,\\
    $d_{\textbf{p}}(8)=8-5=3 $& since $s_1+1=5<8<9=s_2$,\\
    $d_{\textbf{p}}(10)=1$ &since $s_2+1=10$,\\
    $d_{\textbf{p}}(11)=1$ & since $s_3+1=11$, and\\
    $d_{\textbf{p}}(12)=12-10=2$ &since $12=m$.
\end{tabular}
\end{center}
We now count the number of weakly increasing $(12,17)$-parking functions with lucky set $I=\{1,3,5,8,10,11,12\}$ and outcome $\pi=1234\X56789\X10\X(11)(12)\X\X$ in two ways.
First, 
using \Cref{lem:possible prefs given an outcome - mn version}, we have that 
\begin{align}
\label{eq:prod=24}\prod_{i=1}^{n=17}|\Pref_{I}(\pi_i)|=\prod_{\pi_j\notin I\cup\{\X\}}\ell(\pi_j)=
 1\cdot3 \cdot1\cdot2\cdot4=24.\end{align}
Now consider the product
\begin{align}
\displaystyle\frac{\displaystyle\prod_{i=1}^{4}(p_i-1)!}{\displaystyle \prod_{i\in I}d_{\textbf{p}}(i)}=\frac{3!~4!~0!~2!}{1\cdot2\cdot1\cdot3\cdot1\cdot1\cdot2}=24.\label{eq:24 again}
\end{align}
Hence, \Cref{eq:prod=24} and \Cref{eq:24 again} give the same count. Next, we prove the that the formula given in \Cref{eq:24 again} holds in more generality.
\end{example}
In what follows we let 
\[\out^{-1, \uparrow}_{m,n}(I,\pi)=\{\alpha\in\PF_{m,n}^\uparrow:\Lucky(\alpha)=I\mbox{ and }\out(\alpha)=\pi\} 
    \]
    be the set of weakly increasing $(m,n)$-parking functions with  lucky set $I$ and outcome $\pi$.
    
\begin{lemma}\label{lem:using comps for counting pfs with fixed outcome}
    Fix $\pi=\pi_1\pi_2\cdots\pi_n\in\Sym_{m,n}$ an outcome of a weakly increasing  $(m,n)$-parking function with lucky set $I$. Let  $\textbf{p}=(p_1,p_2,\ldots,p_k)$ be the  composition associated to $\pi$ and $\textbf{s}=(0,s_1,s_2,\ldots,s_k)$ be the partial sum list of $\textbf{p}$. 
    Then the number of $\alpha\in\PF_{m,n}^\uparrow$ with lucky set $I$ and outcome $\pi$ is given by 
    \begin{align}\label{eq:prof formula with fracttion}
    |\out^{-1,\uparrow}_{m,n}(I,\pi)|=\displaystyle\frac{\prod_{i=1}^{k}(p_i-1)!}{\prod_{i\in I}d_{\textbf{p}}(i)},\end{align}
    where $0!=1$.
\end{lemma}
\begin{proof}
    Fix $\pi=\pi_1\pi_2\cdots\pi_n\in\Sym_{m,n}$ an outcome of a weakly increasing $(m,n)$-parking function with lucky set $I$.
    The lengths of the occupied sub-streets on which the cars park is given by the composition associated to $\pi$, which we denote by $\textbf{p}=(p_1,p_2,\ldots,p_k)$. In an occupied sub-street of length $p_i$, as the cars park in order, the first has preference for a single spot and is lucky, as it is the first car parking in this occupied sub-street. 
    Of the remaining $p_i-1$ cars parking in the same occupied sub-street, without accounting for which cars are lucky, simply using the fact that they park in increasing order, can have a total of $(p_i-1)!$ options for their preference so as to end up parking in this sub-street. 
    Namely, the $j$th car parked on this sub-street can prefer any of the occupied spots to its left, of which there are $j-1$, and taking the product of $j-1$ over $2\leq j\leq p_i$ yields $1\cdot 2 \cdot 3 \cdots (p_i-1)=(p_i-1)!$. 
    Repeating this process over every occupied sub-street yields the numerator $\prod_{i=1}^k(p_i-1)!$ appearing in \Cref{eq:prof formula with fracttion}.

    However, there may be cars parked on occupied sub-streets which are lucky and are not parked at the start of a sub-street. 
    For any such car, we must reduce their number of preferences down to one possible preference, exactly the location of where they parked. 
    So considering each lucky car $i\in I$ in order from smallest to largest index, by \Cref{def:denoms}, the value $d_{\textbf{p}}(i)$ is the number of preferences we gave lucky car $i$ in the product 
$\prod_{i=1}^k(p_i-1)!$. Hence, for each lucky car $i\in I$, we must divide by $d_{\textbf{p}}(i)$, to remove those preferences from the count. 
Taking the product of $d_{\textbf{p}}(i)$ over all lucky cars reduces the preferences of those lucky cars back to one, as is necessary for them to be lucky. 

Therefore, the number of weakly increasing $(m,n)$-parking functions with outcome $\pi$ and lucky set $I$ is given by 
\[\displaystyle\frac{\prod_{i=1}^{k}(p_i-1)!}{\prod_{i\in I}d_{\textbf{p}}(i)},\]
as claimed.  
\end{proof}

We now give our final result, a product formula for the number of weakly increasing $(m,n)$-parking functions with a fixed set of lucky cars.

\begin{theorem}\label{thm:weakly increasing pfs with a fixed set of lucky cars - mn version}
     Let $1\leq  m\leq n$ and fix the set of lucky cars $I=\{1=i_1<i_2<\cdots<i_k\}\subseteq[m]$. 
    Let $\textbf{i}=(i_2-i_1,i_3-i_2,\ldots,i_{k-1}-i_{k-2},i_{k}-i_{k-1},m+1-i_k)$ and let $\mathcal{S}_{x}(\textbf{i})$ denote the set of good compositions of $\textbf{i}$ with exactly $x+1$ parts, where $0\leq x\leq m-n$.
Then the number of weakly increasing $(m,n)$-parking functions with lucky set $I$ is
\begin{align}\label{eq:cor to theorem}
|\LPF_{m,n}^{\uparrow}(I)|
&=\sum_{\pi=\pi_1\pi_2\cdots\pi_n\in \out_{m,n}^\uparrow(I)}\left(\prod_{\pi_j\notin I\cup\{\X\}} \ell (\pi_j)\right)\\
\label{eq:product formula for number of weakly inc mn pfs with fixed lucky set}
&=\sum_{x=0}^{n-m}\left[\sum_{\textbf{p}=(p_1,p_2,\ldots,p_{x+1})\in \mathcal{S}_{x}(\textbf{i})}\binom{n-m+1}{n-m-x}\left(\displaystyle\frac{\prod_{i=1}^{x+1}(p_i-1)!}{\prod_{i\in I}d_{\textbf{p}}(i)}\right)\right].
\end{align}    
\end{theorem}
\begin{proof}
\Cref{eq:cor to theorem} follows directly from the definition of $\out_{m,n}^\uparrow(I)$ and Theorem \ref{eq:pf count with fixed lucky set - mn version}.
To establish \Cref{eq:product formula for number of weakly inc mn pfs with fixed lucky set} we make the following observations. 
Let $\textbf{p}$ be a good composition of $\textbf{i}$ in $\mathcal{S}_x(\textbf{i})$. 
By Statement (1) in \Cref{prop:description of all outcomes}, the number of parts of $\textbf{p}$ must be $x+1$ with $0\leq x\leq n-m$. 
By \Cref{lem:using comps for counting pfs with fixed outcome} every outcome $\pi$ with the same associated composition $\textbf{p}$ with $x+1$ parts results in the same number of weakly increasing $(m,n)$-parking functions with outcome $\pi$ and lucky set $I$. 
By \Cref{lem:using comps for counting pfs with fixed outcome}, this number is given~by 
    \[\displaystyle\frac{\prod_{i=1}^{x+1}(p_i-1)!}{\prod_{i\in I}d_{\textbf{p}}(i)}.\]
By Statement (2) in \Cref{prop:description of all outcomes}, 
the number of outcomes $\pi$ with the same associated composition $\textbf{p}$ with $x+1$ parts is given by 
$\binom{n-m+1}{n-m-x}$.
So for each good composition in $\mathcal{S}_{x}(\textbf{i})$ there are 
\begin{align}\label{eq:binom many}
    \binom{n-m+1}{n-m-x}
\end{align}
distinct outcome permutations $\pi$. 
Moreover, for each outcome $\pi$, there are
\begin{align}\label{eq:the needed product}
\displaystyle\frac{\prod_{i=1}^{x+1}(p_i-1)!}{\prod_{i\in I}d_{\textbf{p}}(i)}\end{align}
$(m,n)$-parking functions with outcome $\pi$ and lucky set $I$.
Hence, for a particular good composition the product of \Cref{eq:binom many} and \Cref{eq:the needed product} is
\begin{align}\label{eq:needs a sum}
\binom{n-m+1}{n-m-x}\displaystyle\frac{\prod_{i=1}^{x+1}(p_i-1)!}{\prod_{i\in I}d_{\textbf{p}}(i)}\end{align}
and counts the number of $(m,n)$-parking functions parking with outcome $\pi$ having associated composition $\textbf{p}$ and lucky set $I$.
The result then follows from taking \Cref{eq:needs a sum} and summing over all elements in $\mathcal{S}_x(\textbf{i})$ and over all possible values of $x$, which satisfies 
$0\leq x\leq n-m$. This count is given by
\[\sum_{x=0}^{n-m}\left[\sum_{\textbf{p}=(p_1,p_2,\ldots,p_{x+1})\in \mathcal{S}_{x}(\textbf{i})}\binom{n-m+1}{n-m-x}\left(\displaystyle\frac{\prod_{i=1}^{x+1}(p_i-1)!}{\prod_{i\in I}d_{\textbf{p}}(i)}\right)\right],\]
 and this completes the proof.
\end{proof}

We conclude this section with the following example.

\begin{example}[\Cref{ex:computing good comps} continued]
    Let $m=9$, $n=11$, and $I=\{1,4,5\}\subseteq[9]$. Hence $\textbf{i}=(3,1,5)$ and the good compositions of $\textbf{i}$ are $(9),(4,5),(3,6),(3,1,5)$.
    Next, we consider each good composition and determine what it contributes to \Cref{eq:product formula for number of weakly inc mn pfs with fixed lucky set}.
    \begin{itemize}[leftmargin=.2in]
        \item If $\textbf{p}=(9)$, then $x=0$ and $\textbf{s}=(0,9)$. Hence, $d_{\textbf{p}}(1)=1$, $d_\textbf{p}(4)=4-(0+1)=3$, and $d_\textbf{p}(5)=5-(0+1)=4$. Thus 
        \[\binom{n-m+1}{n-m-x}\left(\displaystyle\frac{\prod_{i=1}^{x+1}(p_i-1)!}{\prod_{i\in I}d_{\textbf{p}}(i)}\right)=\binom{3}{2}\left(\frac{8!}{1\cdot3\cdot4}\right)=3\cdot3360=10080.\]

        \item If $\textbf{p}=(4,5)$, then $x=1$ and $\textbf{s}=(0,4,9)$. Hence, $d_{\textbf{p}}(1)=1$, $d_\textbf{p}(4)=4-(0+1)=3$, and $d_\textbf{p}(5)=1$. Thus 
        \[\binom{n-m+1}{n-m-x}\left(\displaystyle\frac{\prod_{i=1}^{x+1}(p_i-1)!}{\prod_{i\in I}d_{\textbf{p}}(i)}\right)=\binom{3}{1}\left(\frac{3!4!}{1\cdot 3\cdot 1}\right)=6\cdot 24=144.\]

        \item If $\textbf{p}=(3,6)$, then $x=1$ and $\textbf{s}=(0,3,9)$. Hence, $d_{\textbf{p}}(1)=1$, $d_\textbf{p}(4)=1$, and $d_\textbf{p}(5)=5-(3+1)=1$. Thus 
        \[\binom{n-m+1}{n-m-x}\left(\displaystyle\frac{\prod_{i=1}^{x+1}(p_i-1)!}{\prod_{i\in I}d_{\textbf{p}}(i)}\right)=\binom{3}{1}\left(\frac{2!5!}{1\cdot 1\cdot 1}\right)=6\cdot 120=720.\]

        \item If $\textbf{p}=(3,1,5)$, then $x=2$ and  $\textbf{s}=(0,3,4,9)$. Hence, $d_{\textbf{p}}(1)=1$, $d_\textbf{p}(4)=1$, and $d_\textbf{p}(5)=1$. Thus
        \[\binom{n-m+1}{n-m-x}\left(\displaystyle\frac{\prod_{i=1}^{x+1}(p_i-1)!}{\prod_{i\in I}d_{\textbf{p}}(i)}\right)=\binom{3}{0}\left(\frac{2!0!4!}{1\cdot 1\cdot 1}\right)=48.\]
    \end{itemize}
    By \Cref{thm:count pfs with fixed lucky set - mn version}, the number of weakly increasing $(9,11)$-parking functions with lucky set $I=\{1,4,5\}$ is
\begin{align}\label{it works!}
    |\LPF_{9,11}^{\uparrow}(I)|=10080+144+720+48=10992.
\end{align}

Alternatively we can use the data from \Cref{ex:computing good comps}, where we established that 
\[\out_{m,n}^\uparrow(I)=\left\{
\begin{matrix}
\X\X123456789, \X123456789\X, 123456789\X\X,
     \X123\X456789, 123\X\X456789,\\
     123\X456789\X,     \X1234\X56789, 1234\X\X56789, 1234\X56789\X,     123\X4\X56789\end{matrix}\right\}.\]
For each $\pi\in \out_{m,n}^\uparrow(I)$, we compute the product in \Cref{eq:cor to theorem}:
\begin{itemize}[leftmargin=.2in]
    \item If $\pi=\X\X123456789,\X123456789\X $, or $123456789\X\X$, then 
    $\prod_{\pi_j\notin I\cup\{\X\}}\ell(\pi_j)=1\cdot 2\cdot 5\cdot 6\cdot 7\cdot 8=3360$.

    \item If $\pi=\X123\X456789, 123\X\X456789$, or
     $123\X456789\X$, then 
    $\prod_{\pi_j\notin I\cup\{\X\}}\ell(\pi_j)=
    1\cdot 2\cdot 2\cdot3 \cdot4 \cdot 5=240
$.
\item If $\pi=\X1234\X56789, 1234\X\X56789, 1234\X56789\X$, or $123\X4\X56789$, then 
    $\prod_{\pi_j\notin I\cup\{\X\}}\ell(\pi_j)=
    1\cdot 2\cdot 1\cdot2 \cdot3 \cdot4 =48
$.
\end{itemize}
By \Cref{eq:cor to theorem}, we have that 
\begin{align}\label{it works again!}
|\LPF_{9,11}^{\uparrow}(I)|=3(3360)+3(240)+4(48)=10992.\end{align}
The counts in \Cref{it works!} and \Cref{it works again!} agree.
\end{example}

\section{Future Work }\label{sec:future}  

Although we have a characterization of the outcomes of a parking function with a fixed lucky set (\Cref{thm:characterizing outcomes}), the enumeration of these permutations is quite sensitive to the location in which the lucky cars park and providing a general formula for the number of outcomes has proven difficult. Thus, it is an open problem to give an enumeration for the set $|\out_n(I)|$ for any lucky set $I$ of $\PF_n$. In addition one could consider the analogous question in the case where there are more parking spots than cars.

In forthcoming work, Beerbower, Elder, Harris, Martinez, and Shirley consider the case of unit interval parking functions, Fubini rankings, and unit Fubini rankings, providing a variety of results both on the number of lucky cars, and the number of such parking functions with a fixed set of lucky cars.
One could extend these studies to other specific subsets of parking functions to include interval parking functions and $\ell$-interval parking functions \cite{colaric2020interval, lintervalrational}, and even to generalizations such as parking sequences and assortments \cite{MR3593646, Chen_2023}. 

For additional open problems related to parking functions we point the interested reader to the article    by Carlson, Christensen, Harris, Jones, and Ramos Rodr\'iguez \cite{Choose}.

\section*{Acknowledgements}
We thank J.~Carlos Mart\'inez Mori, Melissa Beerbower, and Grant Shirley for initial conversations at the start of this project. P.~E.~Harris and L.~Martinez were supported in part by the National Science Foundation award DMS-2150434. L.~Martinez was also supported by the NSF Graduate Research Fellowship Program under Grant No. 2233066.

\bibliographystyle{plain}
\bibliography{bibliography.bib}

\begin{bibdiv}
\begin{biblist}

\bib{lintervalrational}{article}{
      author={Aguilar-Fraga, Tom\'{a}s},
      author={Elder, Jennifer},
      author={Garcia, Rebecca~E.},
      author={Hadaway, Kimberly~P.},
      author={Harris, Pamela~E.},
      author={Harry, Kimberly~J.},
      author={Hogan, Imhotep~B.},
      author={Johnson, Jakeyl},
      author={Kretschmann, Jan},
      author={Lawson-Chavanu, Kobe},
      author={Mori, J. Carlos~Mart\'{i}nez},
      author={Monroe, Casandra~D.},
      author={Qui{\~{n}}onez, Daniel},
      author={III, Dirk~Tolson},
      author={II, Dwight Anderson~Williams},
       title={Interval and $\ell$-interval rational parking functions},
        date={2024},
     journal={Discrete Math. Theor. Comput. Sci.},
      volume={26},
      number={1},
        note={Permutation Patterns 2023 \#10(2024)},
}

\bib{displacement_hanoi}{article}{
      author={Aguillon, Yasmin},
      author={Alvarenga, Dylan},
      author={Harris, Pamela~E.},
      author={Kotapati, Surya},
      author={Mori, J. Carlos~Martínez},
      author={Monroe, Casandra~D.},
      author={Saylor, Zia},
      author={Tieu, Camelle},
      author={II, Dwight Anderson~Williams},
       title={On parking functions and the tower of hanoi},
        date={2023},
     journal={Amer. Math. Monthly},
      volume={130},
      number={7},
       pages={618\ndash 624},
}

\bib{Choose}{article}{
      author={Carlson, Joshua},
      author={Christensen, Alex},
      author={Harris, Pamela~E.},
      author={Jones, Zakiya},
      author={Ramos~Rodr\'{\i}guez, Andr\'{e}s},
       title={Parking functions: {C}hoose your own adventure},
        date={2021},
        ISSN={0746-8342,1931-1346},
     journal={College Math. J.},
      volume={52},
      number={4},
       pages={254\ndash 264},
         url={https://doi.org/10.1080/07468342.2021.1943115},
}

\bib{Chen_2023}{article}{
      author={Chen, Douglas~M.},
      author={Harris, Pamela~E.},
      author={Mori, J. Carlos~Mart\'{i}nez},
      author={Pab\'{o}n-Cancel, Eric~J.},
      author={Sargent, Gabriel},
       title={Permutation invariant parking assortments},
        date={2023-06},
        ISSN={2710-2335},
     journal={Enumer. Comb. Appl.},
      volume={4},
      number={1},
       pages={Article S2R4},
         url={http://dx.doi.org/10.54550/ECA2024V4S1R4},
}

\bib{colaric2020interval}{article}{
      author={Colaric, Emma},
      author={DeMuse, Ryan},
      author={Martin, Jeremy~L.},
      author={Yin, Mei},
       title={Interval parking functions},
        date={2021},
        ISSN={0196-8858,1090-2074},
     journal={Adv. in Appl. Math.},
      volume={123},
       pages={Paper No. 102129, 17},
         url={https://doi.org/10.1016/j.aam.2020.102129},
}

\bib{colmenarejo2024luckydisplacementstatisticsstirling}{article}{
      author={Colmenarejo, Laura},
      author={Dawkins, Aleyah},
      author={Elder, Jennifer},
      author={Harris, Pamela~E.},
      author={Harry, Kimberly~J.},
      author={Kara, Selvi},
      author={Smith, Dorian},
      author={Tenner, Bridget~Eileen},
       title={On the lucky and displacement statistics of stirling permutations},
        date={2024},
     journal={J. Integer Seq.},
      volume={27},
      number={6},
       pages={Article 24.6.7},
}

\bib{colmenarejo2021counting}{article}{
      author={Colmenarejo, Laura},
      author={Harris, Pamela~E.},
      author={Jones, Zakiya},
      author={Keller, Christo},
      author={Rodr\'{i}guez, Andr\'{e}s~Ramos},
      author={Sukarto, Eunice},
      author={Vindas-Mel\'{e}ndez, Andr\'{e}s~R.},
       title={Counting k-naples parking functions through permutations and the k-naples area statistic},
        date={2021-06},
        ISSN={2710-2335},
     journal={Enumer. Comb. Appl.},
      volume={2},
      number={2},
       pages={Article S2R11},
         url={https://ecajournal.kms-ks.org/Volume2021/ECA2021_S2A11.pdf},
}

\bib{cruz2024discretestatisticsparkingfunctions}{article}{
      author={Cruz, Ari},
      author={Harris, Pamela~E.},
      author={Harry, Kimberly~J.},
      author={Kretschmann, Jan},
      author={McClinton, Matt},
      author={Moon, Alex},
      author={Museus, John~O.},
      author={Redmon, Eric},
       title={On some discrete statistics of parking functions},
        date={2024},
     journal={arXiv:2312.16786},
}

\bib{MR3593646}{article}{
      author={Ehrenborg, Richard},
      author={Happ, Alex},
       title={Parking cars of different sizes},
        date={2016},
        ISSN={0002-9890,1930-0972},
     journal={Amer. Math. Monthly},
      volume={123},
      number={10},
       pages={1045\ndash 1048},
         url={https://doi.org/10.4169/amer.math.monthly.123.10.1045},
}

\bib{franks2023counting}{article}{
      author={Franks, Spencer~J.},
      author={Harris, Pamela~E.},
      author={Harry, Kimberly},
      author={Kretschmann, Jan},
      author={Vance, Megan},
       title={Counting parking sequences and parking assortments through permutations},
        date={2024-06},
        ISSN={2710-2335},
     journal={Enumer. Comb. Appl.},
      volume={4},
      number={1},
       pages={Article S2R2},
         url={https://ecajournal.kms-ks.org/Volume2024/ECA2024_S2A2.pdf},
}

\bib{GesselandSeo}{article}{
      author={Gessel, Ira},
      author={Seo, Seunghyun},
       title={A refinement of {C}ayley's formula for trees},
        date={200508},
     journal={Electron. J. Comb.},
      volume={11},
       pages={R27},
}

\bib{MVP}{article}{
      author={Harris, Pamela~E.},
      author={Kamau, Brian~M.},
      author={Mart\'inez~Mori, J.~Carlos},
      author={Tian, Roger},
       title={On the outcome map of {MVP} parking functions: permutations avoiding 321 and 3412, and {M}otzkin paths},
        date={2023},
        ISSN={2710-2335},
     journal={Enumer. Comb. Appl.},
      volume={3},
      number={2},
       pages={Paper No. S2R11, 16},
}

\bib{harris2023lucky}{article}{
      author={Harris, Pamela~E.},
      author={Kretschmann, Jan},
      author={Mori, J. Carlos~Martínez},
       title={Lucky cars and the {Q}uicksort algorithm},
        date={2024},
     journal={Amer. Math. Monthly},
      volume={131},
      number={5},
       pages={417\ndash 423},
}

\bib{konheim1966occupancy}{article}{
      author={Konheim, Alan~G.},
      author={Weiss, Benjamin},
       title={An occupancy discipline and applications},
        date={1966},
     journal={SIAM J. Appl. Math.},
      volume={14},
      number={6},
       pages={1266\ndash 1274},
}

\bib{OEIS}{misc}{
      author={{OEIS Foundation Inc.}},
       title={The {O}n-{L}ine {E}ncyclopedia of {I}nteger {S}equences},
        date={2023},
        note={Published electronically at \url{http://oeis.org}},
}

\bib{Schumacher_Descents_PF}{article}{
      author={Schumacher, Paul R.~F.},
       title={Descents in parking functions},
        date={2018},
     journal={J. Integer Seq.},
      volume={21},
      number={2},
       pages={Article 18.2.3},
}

\bib{Heesung}{article}{
      author={Shin, Heesung},
       title={A new bijection between forests and parking functions},
        date={2008},
     journal={arXiv:0810.0427},
         url={https://arxiv.org/pdf/0810.0427},
}

\bib{Spiro}{article}{
      author={Spiro, Sam},
       title={Subset parking functions},
        date={2019},
     journal={J. Integer Seq.},
      volume={22},
      number={7},
       pages={Article 19.7.3},
}

\end{biblist}
\end{bibdiv}
\end{document}